\documentclass[11pt]{amsart}
\usepackage{amssymb, latexsym, mathrsfs,   color }
\usepackage[all]{xy}
\usepackage[colorlinks=true, pdfstartview=FitV,linkcolor=blue,citecolor=blue,urlcolor=blue]{hyperref}

    \advance \topmargin by -\headheight
    \advance \topmargin by -\headsep

    \evensidemargin \oddsidemargin
\newtheorem {theorem}    {Theorem}[section]

\newtheorem {lemma}      [theorem]    {Lemma}
\newtheorem {corollary}  [theorem]    {Corollary}
\newtheorem {proposition}[theorem]    {Proposition}

\newcommand{\bb}{\mathbb}
\renewcommand{\rm}{\mathrm}

\newcommand{\cal}{\mathcal}
\newcommand{\GG}{\mathrm{G}}
\newcommand{\GGL}{\mathrm{GL}}
\newcommand{\UU}{\mathrm{U}}
\renewcommand{\sp}{\mathrm{Sp}}
\renewcommand{\o}{\mathrm{O}}
\newcommand{\so}{\mathrm{SO}}
\newcommand{\Fq}{\bb{F}_q}
\newcommand{\fq}{(\bb{F}_q)}
\theoremstyle{definition}

\newcommand{\E}{\mathcal{E}}

\newcommand{\CD}{{\mathcal{D}}}
\newcommand{\CQ}{{\mathcal{Q}}}
\newcommand{\CJ}{{\mathcal {J}}}
\newcommand{\UUl}{\underline}
\newcommand{\e}{{\epsilon'}}

\numberwithin{equation}{section}

\begin{document}

\title{Descents of unipotent cuspidal representations of finite classical  groups}

\date{\today}

\author[Dongwen Liu*]{Dongwen Liu*}

\address{School of Mathematical Science, Zhejiang University, Hangzhou 310027, Zhejiang, P.R. China}

\email{maliu@zju.edu.cn}

\author[Zhicheng Wang]{Zhicheng Wang}

\address{School of Mathematical Science, Zhejiang University, Hangzhou 310027, Zhejiang, P.R. China}

\email{11735009@zju.edu.cn}
\begin{abstract}
Inspired by the Gan-Gross-Prasad conjecture and the descent problem for classical groups, in this paper we study the descents of unipotent cuspidal representations of orthogonal and symplectic groups over finite fields.
\end{abstract}

\maketitle

\section{Introduction}

\subsection{Motivation}

In representation theory, a classical problem is to look for the spectral decomposition
of a representation $\pi$ of a group $G$ restricted to a subgroup $H$. Namely, one asks for which representation $\sigma$ of $H$ has the property that
\[
\rm{Hom}_H(\pi,\sigma)\neq 0,
\]
and what the dimension of this Hom-space is. In general such a restriction problem is hard and may not have reasonable answers. However when $G$ is
a classical group defined over a local field and $\pi$ belongs to a generic Vogan $L$-packet, the local Gan-Gross-Prasad conjecture \cite{GP1, GP2, GGP1} provides explicit answers and is one of the most successful examples concerning with those general questions. To be a little more precise, the multiplicity one property holds for this situation, namely
\[
m(\pi,\sigma):=\dim\rm{Hom}_H(\pi,\sigma)\leq 1,
\]
and the invariants attached to $\pi$ and $\sigma$ that detect the multiplicity $m(\pi,\sigma)$ is the local root number associated to their Langlands parameters.  In the $p$-adic case, the local Gan-Gross-Prasad conjecture  has been resolved by J.-L. Waldspurger and C. M\oe glin and J.-L. Waldspurger \cite{W1, W2, W3, MW} for orthogonal groups, by R. Beuzart-Plessis \cite{BP1, BP2} and W. T. Gan and A. Ichino \cite{GI} for unitary groups, and by H. Atobe \cite{Ato} for symplectic-metaplectic groups. On the other hand, D. Jiang and L. Zhang \cite{JZ1} study the local descents for $p$-adic orthogonal groups, whose results can be viewed as a refinement of the local Gan-Gross-Prasad conjecture, and the descent method has important applications towards the global problem (see \cite{JZ2}).

 In a previous work \cite{LW2}, we have studied the descents of unipotent representations of finite unitary groups, applying Reeder's branching formula \cite{R}. The aim of this paper is to study the descent problem for unipotent cuspidal representations of finite symplectic groups and orthogonal  groups, and our main tool is the theta correspondence over finite fields. In a recent paper \cite{P2}, Pan determines  the theta correspondence between finite symplectic and even orthogonal groups. A complete understanding of the theta correspondence should extend our results to more general representations.

 To begin with, we first set up some notations. Let $\overline{\mathbb{F}}_q$ be an algebraic closure of a finite field $\mathbb{F}_q$, which is of characteristic $p>2$. Consider a connected reductive algebraic group $G$ defined over $\Fq$, with Frobenius map $F$. Let $Z$ be the center of $G^F$. We will assume that $q$ is large enough such that the main theorem in \cite{S} holds, namely assume that
 \begin{itemize}
 \item $T^F/Z$ has at least two Weyl group orbits of regular characters, for every $F$-stable maximal torus $T$ of $G$.
 \end{itemize}
For an $F$-stable maximal torus $T$ of $G$ and a character $\theta$ of $T^F$,  let $R_{T,\theta}^G$ be the virtual character of $G^F$ defined by P. Deligne and G. Lusztig in  \cite{DL}. An irreducible representation $\pi$ of $G^F$ is called unipotent if there is an $F$-stable maximal torus $T$ of $G$ such that $\pi$ appears in $R_{T,1}^G$. For two representations $\pi$ and $\pi'$ of a  finite group $H$, define
\[
\langle \pi,\pi'\rangle_H:=\dim \mathrm{Hom}_H(\pi,\pi').
\]
%In order not to cause confusion, we denote $R_{T,\theta}^G$ by $R_{T,\theta}^{\GGL_n(\Fq)}$ (or $R_{T,\theta}^{\UU_n(\Fq)}$) if $G^F=\GGL_n(\Fq)$ (or $\UU_n(\Fq)$).
% We simply write $\langle \pi, \pi'\rangle$ for  $\langle \pi,\pi'\rangle_H$ when no confusion arises.

In this paper, we focus on orthogonal and symplectic groups over finite fields. Let $V$ be an $\Fq$-vector space endowed with a nondegenerate bilinear form $(,)$ with sign $\epsilon$, i.e. $(v, w)=\epsilon (w,v)$ for any $v, w\in V$. Moreover, suppose that $W\subset V$ is a non-degenerate subspace satisfying:
\begin{itemize}

\item $\epsilon \cdot (-1)^{\rm{dim} W^\bot} = -1$,

\item $W^\bot$ is a split space.
\end{itemize}

Then we have
\[
\dim W^\bot=\left\{
\begin{array}{ll}
 \textrm{odd,} & \textrm{if }\epsilon=1,\textrm{ i.e. } V\textrm{ is orthogonal;}\\
\textrm{even,}  & \textrm{if }\epsilon=-1,\textrm{ i.e. } V\textrm{ is symplectic.}
\end{array}\right.
\]
Let $G(V )$ be the identity component of the automorphism group of $V$ and $G(W )\subset G(V )$ the subgroup which acts as identity on $W^\bot$. Let $\pi$ and $\pi'$ be  irreducible representations of $G(V )$ and $G(W )$ respectively.
%Let $P$ be a maximal parabolic subgroup of $U_{n+1}(q)$ with Levi factor $GL_{r}(q^2) \times U_{m}(q)$ (so that $m + 2\ell = n + 1$) and let $\tau$ be a cuspidal representation of $GL_{r}(q^2)$.
The Gan-Gross-Prasad conjecture is concerned with the multiplicity
\[
m(\pi, \pi'):=\langle \pi\otimes\bar{\nu},\pi' \rangle_{H(\Fq)} = \dim \mathrm{Hom}_{H(\Fq)}(\pi\otimes\bar{\nu},\pi' )
\]
where the datum $(H, \nu)$ is defined as in \cite[Theorem 15.1]{GGP1}, and will be explained in details shortly. According to whether $\rm{dim}V-\rm{dim}W$ is odd or even, the above-Hom space is called the Bessel model or Fourier-Jacobi model. In the case of finite unitary groups, W. T. Gan, B. H. Gross and D. Prasad (\cite[Proposition 5.3]{GGP2}) showed  that if  $\pi$ and $\pi'$ are both cuspidal, then
\[
m(\pi,\pi') \le 1.
\]
We should mention that our formulation of multiplicities differs slightly from that in the Gan-Gross-Prasad conjecture \cite{GGP1}, up to taking the contragradient of $\pi'$. This is more suitable for the purpose of descents (c.f. \cite{LW2}), which will be clear from the discussion below. 
%On the other hand, in this paper we will restrict our attention to unipotent cuspidal representations of $\rm{SO}_{n}(\Fq)$ and $\sp_{2n}\fq$, which are self-dual (see Proposition \ref{so3}), and thus for $\pi$ and $\pi'$ unipotent cuspidal the above Hom-space coincides with $\rm{Hom}_{H(\Fq)}(\pi\otimes\pi',\nu)$.

Roughly speaking,  for fixed $G(V)$ and its representation $\pi$,  the descent problem seeks the smallest member $G(W)$ among a Witt tower which has an irreducible representation $\pi'$ satisfying $m(\pi, \pi')\neq 0$, and all such $\pi'$ give the first descent of $\pi$. 
To give the precise notion of descent, we will sketch  the definition of the data $(H,\nu)$ following \cite{GGP1} and \cite{JZ1}.

\subsection{Bessel descent}

We first consider the Bessel case. Let $V_n$ be an $n$-dimensional space over $\mathbb{F}_{q}$ with a nondegenerate symmetric bilinear form $(,)$, which defines the special orthogonal group $\so(V_n)$. We will consider various pairs of symmetric spaces $V_n\supset V_{n-2\ell}$ and the following partitions of $n$,
\begin{equation}\label{pl}
\UUl{p}_\ell=[2\ell+1, 1^{n-2\ell-1}],\quad 0\leq \ell\leq n/2.
\end{equation}
Assume that  $V_n$ has a decomposition
\[
V_n=X+V_{n-2\ell}+X^\vee
\]
where $X+X^\vee=V_{n-2\ell}^\perp$ is a polarization. Let $\{e_1,\ldots, e_\ell\}$ be a basis of $X$, $\{e_1',\ldots, e_\ell'\}$ be the dual basis of $X^\vee$, and
let $X_i=\rm{Span}_{\mathbb{F}_{q}}\{e_1,\ldots, e_i\}$, $i=1,\ldots, \ell$. Let $P$ be the parabolic subgroup of $\so(V_n)$ stabilizing the flag
\[
X_1\subset\cdots\subset X_\ell,
\]
so that its Levi component is $M\cong \GGL_1^\ell\times \so(V_{n-2\ell})$. Its unipotent radical can be written in the form
\[
N_{\UUl{p}_\ell}=\left\{n=
\begin{pmatrix}
z & y & x\\
0 & I_{n-2\ell} & y'\\
0 & 0 & z^*
\end{pmatrix}
: z\in U_{\GGL_\ell}
\right\},
\]
where the superscript ${}^*$ denotes the  transpose inverse, and $U_{\GGL_\ell}$ is the subgroup of unipotent upper triangular matrices of $\GGL_\ell$.
Fix a nontrivial additive character $\psi$ of $\mathbb{F}_{q}$. Pick up an anisotropic vector $v_0\in V_{n-2\ell}$ and define a generic character $\psi_{\UUl{p}_\ell, v_0}$ of $N_{\UUl{p}_\ell}(\Fq)$ by
\[
\psi_{\UUl{p}_\ell, v_0}(n)=\psi\left(\sum^{\ell-1}_{i=1}z_{i,i+1}+(y_\ell, v_0)\right), \quad n\in N_{\UUl{p}_\ell}(\Fq),
\]
where $y_\ell$ is the last row of $y$. The identity component of the stabilizer of $\psi_{\UUl{p}_\ell, v_0}$ in $M(\Fq)$ is the special orthogonal group $\rm{SO}(W)$, where $W$ is the orthogonal complement of $v_0$ in $V_{n-2\ell}$.
Put
\begin{equation}\label{hnu}
H=\so(W)\ltimes N_{\UUl{p}_\ell},\quad \nu=\psi_{\UUl{p}_\ell, v_0}.
\end{equation}
Let $\pi$ and $\pi'$ be two irreducible cuspidal representations of $\so(V_{n})$ and $\so(W)$ respectively. Since depth-zero supercuspial representations of $p$-adic Lie groups are compactly induced from cuspidal representations of finite Lie groups,
the uniqueness of
{\sl Bessel models} in the $p$-adic case proved in \cite{AGRS} readily implies  that
\[
m(\pi,\pi'):=\dim_{H(\Fq)}(\pi\otimes \bar{\nu}, \pi')\leq 1.
\]

Note that $m(\pi,\pi')$ depends on the choice of $v_0$. Let $Q$ be the quadratic form associated to $(,)$.
Pick up two anisotropic vectors $v_0$, $v_0'\in V_{n-2\ell}$ such that $Q(v_0)/Q(v_0')$ is a non-square in $\Fq$. The identity component of the stabilizer of $\psi_{\UUl{p}_\ell, v_0'}$ in $M(\Fq)$ is the special orthogonal group $\so(W')$ of the orthogonal complement $W'$ of $v_0'$ in $V_{n-2\ell}$. If $n-2\ell$ is even, then $\so(W)\cong \so(W')$, but the groups $\so(W)$, $\so(W')$ are not conjugate in $\so(V_{n-2\ell})$.
If $n-2\ell$ is odd, then there are two choices of anisotropic vectors $v_0$, $v_0'\in V_{n-2\ell}$ such that $W$ is split but $W'$ not. Thus we get $\so(W)\ncong \so(W')$ in this case. In general, we have
\begin{equation}\label{disc}
\rm{disc} \ V= (-1)^{n-1} \cdot Q(v_0) \cdot \rm{disc} \ W,
\end{equation}
where both sides are regarded as square classes in $\Fq^\times/ (\Fq^\times)^2\cong\{\pm1\}$. Here the discriminant is normalized by
\[
\rm{disc} \ V=(-1)^{n(n-1)/2}\det V \in \Fq^\times/ (\Fq^\times)^2,
\]
such that when $\dim V$ is even, $\rm{disc} \ V=+1$ if and only if $\rm{SO}(V)$ is split.

%Let $\pi$ and $\pi'$ be two irreducible representations of $\so(V_{n})$ and $\so(W)$ respectively. Then $\pi'$ is also a representation of $\so(W')$. We denote the multiplicity $m(\pi,\pi')$ of two different choices of anisotropic
%vector in $V_{n-2\ell}$ by $m^\pm(\pi,\pi')$, respectively.

Let $\CJ_{\ell,{v_0}}(\pi)$ be the twisted Jacquet module of $\pi$ with respect to $(N_{\UUl{p}_{\ell}}(\Fq), \psi_{\UUl{p}_{{\ell}}, v_0})$.
We simply define the notion of the ${\ell}$-th {\sl Bessel quotient} of $\pi$ with respect to (the rational orbit of) $v_0$ by
\begin{equation}\label{lbd}
\CQ^\rm{B}_{{\ell},{v_0}}(\pi):=\CJ_{{\ell},{v_0}}(\pi),
\end{equation}
which is a representation of $\so(W)$. Define the {\sl first occurrence index} ${\ell}_0:={\ell}_0^\rm{B}(\pi)$ of $\pi$ in the Bessel case to be the largest nonnegative integer ${\ell}_0 \leq n/2$ such that $\CQ^\rm{B}_{{\ell}_0,{v_0}}(\pi)\neq 0$ for some anisotropic vector $v_0\in V_{n-2{\ell}_0}$. The ${\ell}_0$-th Bessel descent of $\pi$ with respect to this choice of $v_0$ is called the {\sl first Bessel descent} of $\pi$ or simply the {\sl Bessel descent} of $\pi$, denoted by
\begin{equation}\label{bd}
\CD^\rm{B}_{{\ell}_0, v_0}(\pi):=\CQ^\rm{B}_{{\ell}_0,{v_0}}(\pi).
\end{equation}

The above discussions are valid for full orthogonal groups as well. For an irreducible representation $\pi$ of $\o(V_n)$, we have the $\ell$-th Bessel descent $\CQ^\rm{B}_{\ell, v_0}(\pi)$ as a representation of $\o(W)$, and we also have the notions of the first occurrence index $\ell_0^\rm{B}(\pi)$ and the first Bessel descent $\CD^\rm{B}_{\ell_0, v_0}(\pi)$.

\subsection{Fourier-Jacobi descent}

We next turn to the Fourier-Jacobi case. Let $W_{2n}$ be a symplectic space of dimension $2n$ over $\Fq$, which gives the symplectic group $\sp_{2n}(\Fq)$. Consider pairs of symplectic spaces $W_{2n}\supset W_{2n-2\ell}$ and partitions
\begin{equation}\label{pl'}
\UUl{p}'_\ell=[2\ell, 1^{2n-2\ell}],\quad 0\leq \ell\leq n.
\end{equation}
We use similar notations for various subspaces and subgroups as in the Bessel case. Note that if we let $P_\ell$ be the parabolic subgroup of $\rm{Sp}_{2n}$ stabilizing $X_\ell$ and let $N_\ell$ be its unipotent radical, then $N_{\UUl{p}_\ell}=U_{\GGL_\ell}\ltimes N_\ell$. Let $\omega_\psi$ be the Weil representation (see \cite{Ger}) of $\sp_{2(n-\ell)}(\Fq)\ltimes \mathcal{H}_{2n-2\ell}$ depending on $\psi$, where $\mathcal{H}_{2n-2\ell}$ is the Heisenberg group of $W_{2n-2\ell}$. Roughly speaking, there is a natural homomorphism $N_\ell(\Fq)\to \mathcal{H}_{2n-2\ell}$ invariant under the conjugation action of $U_{\GGL_\ell}(\Fq)$ on $N_\ell(\Fq)$, which enables us to view $\omega_\psi$ as a representation of $\sp_{2(n-\ell)}(\Fq)\ltimes N_{\UUl{p}_\ell}(\Fq)$. Let $\psi_\ell$ be the character of $U_{\GGL_\ell}(\Fq)$ given by
\[
\psi_\ell(z)=\psi\left(\sum^{\ell-1}_{i=1}z_{i,i+1}\right),\quad z\in U_{\GGL_\ell}(\Fq).
\]
For the Fourier-Jacobi case, put
\begin{equation}\label{hnu'}
H=\sp_{2(n-\ell)}\ltimes N_{\UUl{p}_\ell},\quad \nu=\omega_\psi\otimes\psi_\ell.
\end{equation}
Similar to the Bessel case, for irreducible cuspidal representations $\pi$ and $\pi'$ of $\sp_{2n}(\Fq)$ and $\sp_{2(n-\ell)}(\Fq)$ respectively, the uniqueness of {\sl Fourier-Jacobi models}  in the $p$-adic case proven  in \cite{Su} implies that
 \[
m_\psi(\pi,\pi'):=\rm{Hom}_{H(\Fq)}(\pi\otimes \bar{\nu}, \pi')\leq 1.
 \]

Note that $m_\psi(\pi, \pi')$ depends on $\psi$, and that
\[
\rm{Hom}_{H(\Fq)}(\pi\otimes \bar{\nu}, \pi')\cong \rm{Hom}_{\rm{Sp}_{2n-2\ell}(\Fq)}(\CJ'_\ell(\pi\otimes\overline{\omega_\psi}), \pi'),
\]
where $\CJ'_\ell(\pi\otimes\overline{\omega_\psi})$ is the twisted Jacquet module of $\pi\otimes\overline{\omega_\psi}$ with respect to $(N_{\UUl{p}_\ell}(\Fq), \psi_\ell)$. Define the $\ell$-th {\sl Fourier-Jacobi quotient} of $\pi$ with respect to $\psi$ to be
\begin{equation}\label{lfd}
\CQ_{\ell,\psi}^\rm{FJ}(\pi):=\CJ'_\ell(\pi\otimes\overline{\omega_\psi}),
\end{equation}
viewed as a representation of $\sp_{2(n-\ell)}(\Fq)$. Define the {\sl first occurrence index} $\ell_0:=\ell_0^\rm{FJ}(\pi)$ of $\pi$ in the Fourier-Jacobi case to be the largest nonnegative integer $\ell_0\leq n/2$ such that $\CQ^\rm{FJ}_{\ell_0,\psi}(\pi)\neq 0$ for some choice of $\psi$. The $\ell_0$-th Fourier-Jacobi descent of $\pi$ with respect to this chosen $\psi$ is called the {\sl first Fourier-Jacobi descent} of $\pi$ or simply the {\sl Fourier-Jacobi descent} of $\pi$, denoted by
\begin{equation}\label{bd}
\CD^\rm{FJ}_{\ell_0, \psi}(\pi):=\CQ^\rm{FJ}_{\ell_0,\psi}(\pi).
\end{equation}

\subsection{Cuspidal unipotent and $\theta$-representations}

We now fix some  notations and  describe the cuspidal representations considered in this paper. As is standard, denote by $\so_n^\epsilon$ and $\o_n^\epsilon$, $\epsilon=\pm$,  the (special) orthogonal groups of an $n$-dimensional  quadratic space with discriminant $\epsilon \ 1\in \mathbb{F}_q^\times /(\mathbb{F}_q^\times)^2$.   For convenience, by abuse of notation we also write $\epsilon=\epsilon \ 1$ for the sign of the corresponding discriminant. Denote by $\epsilon_a$, $a\in \mathbb{F}_q^\times$ the square class of $a$, so that $\epsilon_{-1}$ is involved in \eqref{disc}. Put $\epsilon(k)=(-1)^k$ for an integer $k$.

According to Lusztig's results \cite{L1}, let $\pi_{\sp_{2k(k+1)}}$, $\pi_{\rm{SO}^\epsilon_{2k(k+1)+1}}$, $\epsilon=\pm$, and $\pi_{\rm{SO}^{\epsilon(k)}_{2k^2}}$ be the unique irreducible unipotent cuspidal representations of the corresponding groups. The irreducible unipotent cuspidal representations of $\o_n^\epsilon$, where $\so_n^\epsilon$ is one of the above special orthogonal  groups, are the two irreducible components of 
\begin{equation}\label{un}
\rm{Ind}^{\o^\epsilon_n}_{\rm{SO}^\epsilon_n}\pi_{\so_n^\epsilon}=\pi^+_{\o_n^\epsilon} \oplus \pi^-_{\o_n^\epsilon}.
\end{equation}
Note that $\pi^+_{\o_n^\epsilon}$ and $\pi^-_{\o_n^\epsilon}$ differ by the sign character of $\o_n^\epsilon$, and that 
\begin{equation}\label{res}
\pi^\pm_{\o_n^\epsilon}\vert_{\so_n^\epsilon}\cong  \pi_{\so_n^\epsilon}.
\end{equation} We distinguish them by decreeing that
\begin{itemize}
\item when $n=2k(2k+1)$ is odd, $\pi^\pm_{\o_n^\epsilon}(-1)=\pm \ \textrm{Id}$;
\item $\pi^{+}_{\o_2^-}= \textrm{triv}$, $\pi^-_{\o_2^-}=\textrm{sgn}$, and $\pi^{\pm}_{\o_{2k^2}^{\epsilon(k)}}$, $k\geq 2$ are determined by the chain of conservation relations as in \cite{AM}. See Theorem \ref{even} for details. 
\end{itemize}
In a previous work \cite{LW1}, we introduced a notion of $\theta$-representations (which are called pseudo-unipotent representations in \cite{P2}) in order to study the theta correspondence between finite symplectic and odd orthogonal groups.  Based on Lusztig's results, in \cite{LW1} we proved that  
$\rm{Sp}_{2n}$, $n=k^2$ are the only symplectic groups which possess cuspidal $\theta$-representations, and each $\rm{Sp}_{2k^2}$ has two irreducible cuspidal $\theta$-representations $\pi^\theta_{k,\alpha}$ and $\pi^\theta_{k,\beta}$, which satisfy $\pi^\theta_{k,i}(-1)=(-1)^{k}\cdot\rm{Id}$, $i=\alpha,\beta$. We distinguish them as follows. 

When $k=1$, $\pi^\theta_{1,\alpha}$ and $\pi^\theta_{1,\beta}$ are generic representations of $\textrm{SL}_2\fq$ with respect to non-conjugate generic data. Recall that $\psi$ is a  nontrivial additive character of $\mathbb{F}_q$, which will be fixed throughout the paper. Let $\psi'$ be another nontrivial additive character of $\mathbb{F}_q$ not in the square class of $\psi$. We label the cuspidal $\theta$-representations in the way  that 
\begin{itemize}
\item $\pi^\theta_{1,\alpha}$ and $\pi^\theta_{1,\beta}$ occur in the Weil representations 
$\omega_{\textrm{SL}_2,\psi}$ and $\omega_{\textrm{SL}_2, \psi'}$ of $\textrm{SL}_2\fq$, respectively;
\item $\pi^\theta_{k, i}$, $i=\alpha, \beta$, $k\geq 2$ are determined by the chain of conservation relations as in \cite{LW1}. See Theorem \ref{spthetalift} for details.
\end{itemize}

\subsection{Main result}

In the above settings, we now present the main result of this paper. 

\begin{theorem}\label{main}
(i) For an irreducible unipotent cuspidal representation  $\pi^\eta_{\rm{O}^\epsilon_{2k(k+1)+1}}$ of $\rm{O}^\epsilon_{2k(k+1)+1}\fq$, one has $\ell_0^\rm{B}(\pi^\eta_{\rm{O}^\epsilon_{2k(k+1)+1}})=k$ and 
 \[
 \CD^\rm{B}_{k, v_0}(\pi^\eta_{\rm{O}^{\epsilon}_{2k(k+1)+1}})=\pi^{\eta \cdot \epsilon(k)}_{\rm{O}^{\epsilon(k)}_{2k^2}}, 
\]
where $Q(v_0)= \epsilon \cdot \epsilon(k)$.

(ii) For an irreducible unipotent cuspidal representation $\pi^\eta_{\rm{O}^{\epsilon(k)}_{2k^2}}$ of $\rm{O}^{\epsilon(k)}_{2k^2}\fq$,  one has 
$\ell_0^\rm{B}(\pi_{\rm{O}^{\epsilon(k)}_{2k^2}})=k-1$ and
\[
 \CD^\rm{B}_{k-1, v_0}(\pi^\eta_{\rm{O}^{\epsilon(k)}_{2k^2}})=\pi^{\eta\cdot \epsilon(k-1)}_{\rm{O}^{\epsilon}_{2k(k-1)+1}},
\]
where $Q(v_0)=\epsilon_{-1}\cdot \epsilon\cdot\epsilon(k)$.

(iii) For the unique irreducible unipotent cuspidal representation $\pi_{\sp_{2k(k+1)}}$ of $\sp_{2k(k+1)}\fq$, one has  $\ell_0^\rm{FJ}(\pi_{\sp_{2k(k+1)}})=k$ and 
\[
 \CD^\rm{FJ}_{k, \psi}(\pi_{\sp_{2k(k+1)}})=\pi^\theta_{k,\alpha_k}, \quad  \CD^\rm{FJ}_{k, \psi'}(\pi_{\sp_{2k(k+1)}})=\pi^\theta_{k,\beta_k},
\]
where $(\alpha_k, \beta_k)=(\alpha, \beta)$ or $(\beta, \alpha)$ for $\epsilon_{-1}\cdot \epsilon(k)=+1$ or $-1$, respectively. 
\end{theorem}

As an immediate consequence of  \eqref{res}, we have the following Bessel descent for special orthogonal groups.

\begin{corollary}
(i) For the unique irreducible unipotent cuspidal representation  $\pi_{\rm{SO}^\epsilon_{2k(k+1)+1}}$ of $\rm{SO}^\epsilon_{2k(k+1)+1}\fq$, one has $\ell_0^\rm{B}(\pi_{\rm{SO}^\epsilon_{2k(k+1)+1}})=k$ and 
 \[
 \CD^\rm{B}_{k, v_0}(\pi_{\rm{SO}^{\epsilon}_{2k(k+1)+1}})=\pi_{\rm{SO}^{\epsilon(k)}_{2k^2}}, 
\]
where $Q(v_0)= \epsilon \cdot \epsilon(k)$.

(ii) For the unique irreducible unipotent cuspidal representation $\pi_{\rm{SO}^{\epsilon(k)}_{2k^2}}$ of $\rm{SO}^{\epsilon(k)}_{2k^2}\fq$,  one has 
$\ell_0^\rm{B}(\pi_{\rm{SO}^{\epsilon(k)}_{2k^2}})=k-1$ and
\[
 \CD^\rm{B}_{k-1, v_0}(\pi_{\rm{SO}^{\epsilon(k)}_{2k^2}})=\pi_{\rm{SO}^{\epsilon}_{2k(k-1)+1}},
\]
where $Q(v_0)=\epsilon_{-1}\cdot \epsilon\cdot\epsilon(k)$.
\end{corollary}

 This paper is organized as follows. In Section \ref{sec2}, we recall the notion of Harish-Chandra series. In Section \ref{sec3}, we recall the theory of Weil representation, theta correspondence and see-saw dual pairs. In Section \ref{sec4} we  recall the theta correspondence and the first occurrence index of unipotent cuspidal representations of finite orthogonal groups and symplectic groups. In Section \ref{sec5} we  prove the Bessel case of Theorem \ref{main}. In Section \ref{sec6} we prove the Fourier-Jacobi case.

\subsection*{Ackonwledgement} We thank the anonymous referee for raising numerous comments which improve the exposition of this paper. 

\section{Harish-Chandra series}\label{sec2}
Let $G$ be a reductive group defined over $\Fq$, $F$ be the corresponding Frobenius endomorphism, and $\E(G)=\rm{Irr}(G^F)$ be the set of irreducible representations of $G^F$.
A parabolic subgroup $P$ of $G$ is the normalizer in $G$ of a parabolic subgroup $P^\circ$ of the connected component $G^\circ$ of $G$.
A Levi subgroup $L$ of $P$ is the normalizer in $G$ of the  Levi subgroup $L^\circ$ of $P^\circ$. Then we have a Levi decomposition $P=LV$. If $P$ is $F$-stable, then we have $P^F = L^FV^F.$ Let $\delta$ be a representation of the group $L^F$. We can lift $\delta$ to a character of $P^F$ by making it trivial on $V^F$. We have the parabolic induction
\begin{equation}\label{paraind}
I_L^G(\delta):=I_P^G(\delta)=\mathrm{Ind}_{P^F}^{G^F}\delta.
\end{equation}
It is well-known that the induction in stages holds (see e.g. \cite[Proposition 4.7]{DM}), namely if $Q \subset P$ are two parabolic subgroups of $G$ and $M \subset L$
are the corresponding Levi subgroups, then
\[
I^G_L\circ I^L_M=I^G_M.
\]
We say that a pair $(L, \delta)$ is cuspidal if $\delta$ is cuspidal.

\begin{theorem} For $\pi \in \E(G)$, there is  a unique cuspidal pair $(L, \delta)$ up to $G^F$-conjugacy such that $\langle\pi,I_{L}^G(\delta)\rangle_{G^F}\ne 0$
\end{theorem}

Thus we get a partition of $\E(G)$ into series parametrized by $G^F$-conjugacy classes of cuspidal pairs $(L, \delta)$. The Harish-Chandra series of $(L, \delta)$ is the set of irreducible representations of $G^F$ appearing in $I_{L}^G(\delta)$.
We focus on classical groups, and let $L$ be an $F$-stable standard Levi subgroup of $G_n:=\sp_{2n}$, $\o^\pm_{2n}$ or $\o_{2n+1}$. Then $L^F$ has a standard form
\[
L^F=\GGL_{n_1}\fq\times \GGL_{n_2}\fq\times\cdots\times \GGL_{n_r}\fq\times G_m^F
\]
where $G_m=\sp_{2m}$, $\o^\pm_{2m}$ or $\o_{2m+1}$, and $n_1+\cdots+n_r +m=n$. For a cuspidal pair $(L,\delta)$, one has
\[
\delta=\rho_1\otimes\cdots\otimes\rho_r\otimes\sigma
\]
where $\rho_i$ and $\sigma$ are cuspidal representations of $\GGL_{n_i}\fq$ and $G_m^F$, respectively.

By induction in stages, for any irreducible component $\pi$ of $I_L^G(\delta)$, there exists $\rho\in \E(\GGL_{n-m})$ such that
$\pi \subset I_{\GGL_{n-m}\times G_m}^G(\rho\otimes\sigma).
$
Let
\[
\cal{E}(G_n,\sigma)=\{\pi\in \cal{E}(G_n)|\langle\pi,I_{\GGL_{n-m}\times G_m}^G(\rho\otimes\sigma)\rangle_{G^F}\ne 0\textrm{ for some }\rho\in \E(\GGL_{n-m})\}.
\]
Then we have  a disjoint union
\[
\E(G_n)=\bigcup_{\sigma}\cal{E}(G_n,\sigma),
\]
where $\sigma$ runs over all irreducible cuspidal representations of $G^{F}_m$, $m=0,1,\cdots,n$.

\section{Theta correspondence and see-saw dual pairs}\label{sec3}

As mentioned earlier, we fix the nontrivial additive character $\psi$ of $\Fq$ throughout. Let $\omega_{\rm{Sp}_{2N}}=\omega_{\rm{Sp}_{2N}, \psi}$ be the Weil representation  of the finite symplectic group $\rm{Sp}_{2N}(\Fq)$, which depends on $\psi$.
%Let $\omega^{\#}_{\rm{Sp}_{2n}}$ denote the uniform projection of $\omega_{\rm{Sp}_{2N}}$, i.e. the projection onto the subspace of virtual characters spanned by all the Deligne-Lusztig characters.
Let $(G, G^{\prime})$ be a reductive dual pair in $\rm{Sp}_{2N}$, and write
$\omega_{G,G'}$ for the restriction of $\omega_{\rm{Sp}_{2N}}$ to $G^F\times G'^F$. Then it decomposes into a direct sum
\[
\omega_{G,G'}=\bigoplus_{\pi,\pi'} m_{\pi,\pi '}\pi\otimes\pi',
\]
where $\pi$ and $\pi '$ run over $\rm{Irr}(G^F)$ and $\rm{Irr}(G'^F)$ respectively, and $m_{\pi,\pi'}$ are nonnegative integers. Rearrange this decomposition as
\[
\omega_{G,G'}=\bigoplus_{\pi} \pi\otimes\Theta_{G,G'}(\pi )
\]
 where $\Theta_{G, G'}(\pi ) = \bigoplus_{\pi'} m_{\pi,\pi '}\pi '$ is a (not necessarily irreducible) representation of $G'^F$, called the (big) theta lifting of $\pi$ from $G$ to $G'$. Write $\pi'\subset \Theta_{G'}(\pi)$ if $\pi\otimes\pi'$ occurs in $\omega_{G,G'}$, i.e. $m_{\pi, \pi'}\neq 0$. We remark that even if $\Theta_{G,G'}(\pi)=:\pi'$ is irreducible, one only has
 \[
 \pi\subset \Theta_{G',G}(\pi'),
 \]
 where the equality does not  necessarily hold in general.

It is convenient to work with the families of dual pairs $(G_n, G_{n'}')$  associated to Witt towers $G_n\in \cal{T}$ and $G_{n'}'\in \cal{T}'$ instead of  a single dual pair. In this paper we only consider the following Witt towers.

 $\bullet$ For symplectic groups there is only one Witt tower ${\bf Sp}=\left\{\sp_{2n}\right\}_{n\geq 0}$.

 $\bullet$ For even orthogonal groups there are two Witt towers ${\bf O}^+_\rm{even}=\left\{\o^+_{2n}\right\}_{n\geq 0}$ and ${\bf O}^-_\rm{even}=\left\{\o^-_{2n}\right\}_{n\geq 1}$.

$\bullet$ For odd orthogonal groups there are two Witt towers as well ${\bf O}^\epsilon_\rm{odd}=\left\{\o^\epsilon_{2n+1}\right\}_{n\geq 0}$, $\epsilon=\pm$.

Recall the convention that $\o^+_{2n}$ (resp. $\o^-_{2n}$) denotes the isometry group of the split (resp. nonsplit) form of dimension $2n$. For odd orthogonal groups, one has $\o^+_{2n+1}\cong\o^-_{2n+1}$ as abstract groups; however they act on two quadratic spaces with different discriminants. 

When the context of the pair of Witt towers $\{G_n\}$ and   $\{G'_{n'}\}$ is clear,  write $\omega^\epsilon_{n,n'}$ instead of $\omega_{G_n,G'_{n'}}$, and  denote $\Theta_{n,n'}=\Theta^\epsilon_{n,n'}$ the theta lifting from $G_n$ to $G'_{n'}$, where the superscript $\epsilon$ reminds the discriminant of the orthogonal Witt tower.
For an irreducible representation $\pi$ of $G_n$, the smallest integer $n^\epsilon(\pi)$ such that $\pi$ occurs in $\omega^\epsilon_{n,n^\epsilon(\pi)}$ is called the {\it first occurrence index} of $\pi$ in the Witt tower $\left\{G'_{n'}\right\}$. By \cite[Chap.3, lemme IV.2]{MVW}, there exists $n'$ such that $\Theta_{n, n'}^\epsilon(\pi)\neq 0$, hence $n^\epsilon(\pi)$ is well-defined.  Note that the first occurrence indices depend on the choice of $\psi$, and are subject to various conservation relations. 

The next result shows that the theta lifting and the parabolic induction are compatible.

\begin{proposition}\label{w1}
Let $G_n$ and $G_{n+\ell}$ be two classical groups in the same Witt tower, $\ell\geq 0$. Let $\tau$ be an irreducible cuspidal representation of $\GGL_\ell(\mathbb{F}_{q})$, $\pi$ be an irreducible representation of $G_n(\Fq)$, and  $\pi':=\Theta_{n,n'}(\pi)$. Let $\chi_{\GGL_{\ell}}$ be the unique linear character of $\GGL_\ell\fq$ of order $2$. Let $\rho\subset I^{G_{n+\ell}}_{\GGL_\ell\times G_{n}}(\tau \otimes \pi)$ be an irreducible representation of $G_{n+\ell}$ and $\rho' \subset \Theta_{n+\ell, n'+\ell}(\rho)$ be an irreducible representation of $G'_{n'+\ell}$.  Assume that $\tau$ is non-selfdual if $\ell=1$. Then we have
\[
\rho'\subset I^{G'_{n'+\ell}}_{\GGL_\ell\times G'_{n'}}((\chi\otimes\tau)\otimes \pi'),
\]
where
\[
\chi=\left\{
\begin{array}{ll}
\chi_{\GGL_{\ell}}, &  \textrm{if $(G_{n+\ell},G'_{n'+\ell})$ contains an odd orthogonal group,}\\
1, & \textrm{otherwise.}
\end{array}\right.
\]
In particular, if $I^{G_{n+\ell}}_{\GGL_\ell\times G_{n}}(\tau \otimes \pi)$ is irreducible, then
\[
\Theta_{n+\ell, n'+\ell}(I^{G_{n+\ell}}_{\GGL_\ell\times G_{n}}(\tau \otimes \pi))= I^{G'_{n'+\ell}}_{\GGL_\ell\times G'_{n'}}((\chi\otimes\tau)\otimes \pi').
\]
\end{proposition}
\begin{proof}
We will only prove the proposition for $(G_{n},G_{n'}')\in(\bf{Sp},\bf{O}^\epsilon_{\rm{odd}})$. The proof for other cases is similar and will be left to
the reader.

Here $J$ standards for the Jacquet functor, which is adjoint to the induction functor $I$. We have the following decomposition (cf. \cite[Chap. 3, IV th.5]{MVW})
\[
\begin{aligned}
&J^{\sp_{2(n+\ell)}}_{\sp_{2n}\times \GGL_{\ell}}(\omega^\epsilon_{n+\ell,n'+\ell})
\\=&\bigoplus^{\ell}_{i=0}I^{\sp_{2n}\times \GGL_{\ell} \times \o^\epsilon_{2(n'+\ell)+1}}_{\sp_{2n}\times (  \GGL_{\ell-i}\times \GGL_{i})\times \GGL_{i}\times \o^\epsilon_{2(n'+\ell-i)+1}}(\omega^\epsilon_{n,n'+\ell-i} \otimes \chi_{\GGL_{\ell-i}} \otimes \chi_{\GGL_{i}}R^{\GGL_{i}})
\end{aligned}
\]
where the regular representation $R^{\GGL_{i}}$ is considered as a representation  of $\GGL_i\fq\times \GGL_i\fq$.
Then
\[
\begin{aligned}
&\langle \omega^\epsilon_{n+\ell,n'+\ell}, I^{G_{n+\ell}}_{\GGL_\ell\times G_{n}}(\tau \otimes \pi)\otimes\rho'\rangle\\
=&\langle J^{\sp_{2(n+\ell)}}_{\sp_{2n}\times \GGL_{\ell}}(\omega^\epsilon_{n+\ell,n'+\ell}), (\tau \otimes \pi)\otimes\rho'\rangle\\
=&\bigoplus^{\ell}_{i=0}\langle I^{\sp_{2n}\times \GGL_{\ell} \times \o^\epsilon_{2(n'+\ell)+1}}_{\sp_{2n}\times (  \GGL_{\ell-i}\times \GGL_{i})\times \GGL_{i}\times \o^\epsilon_{2(n'+\ell-i)+1}}(\omega^\epsilon_{n,n'+\ell-i} \otimes \chi_{\GGL_{\ell-i}} \otimes \chi_{\GGL_{i}}R^{\GGL_{i}})  , (\tau \otimes \pi)\otimes\rho'\rangle\\
=& \langle I^{\sp_{2n}\times \GGL_{\ell} \times \o^\epsilon_{2(n'+\ell)+1}}_{\sp_{2n}\times  \GGL_{\ell}\times \GGL_{\ell}\times \o^\epsilon_{2n'+1}}(\omega^\epsilon_{n,n'} \otimes  \chi_{\GGL_{\ell}}R^{\GGL_{\ell}})  , (\tau \otimes \pi)\otimes\rho'\rangle\\
=& \langle (\tau \otimes \pi)\otimes I^{ \o^\epsilon_{2(n'+\ell)+1}}_{\GGL_{\ell}\times \o^\epsilon_{2n'+1}}(\chi_{\GGL_{\ell}}\tau\otimes\Theta^\epsilon_{n,n'}(\pi)) ,(\tau \otimes \pi)\otimes\rho' \rangle\\
=& \langle  I^{ \o^\epsilon_{2(n'+\ell)+1}}_{\GGL_{\ell}\times \o^\epsilon_{2n'+1}}(\chi_{\GGL_{\ell}}\tau\otimes\pi') ,\rho' \rangle.
\end{aligned}
\]
By our assumption, one has
\[
\langle  I^{ \o^\epsilon_{2(n'+\ell)+1}}_{\GGL_{\ell}\times \o^\epsilon_{2n'+1}}(\chi_{\GGL_{\ell}}\tau\otimes\pi') ,\rho' \rangle=\langle \omega^\epsilon_{n+\ell,n'+\ell}, I^{G_{n+\ell}}_{\GGL_\ell\times G_{n}}(\tau \otimes \pi)\otimes\rho'\rangle\ge\langle \omega^\epsilon_{n+\ell,n'+\ell},\rho\otimes\rho'\rangle>0.
\]
\end{proof}

Recall the general formalism of see-saw dual pairs. Let $(G, G')$ and $(H, H')$ be two reductive dual pairs in a finite symplectic group $\rm{Sp}(W)$ such that $H \subset G$ and $G' \subset H'$.  Then there is a see-saw diagram
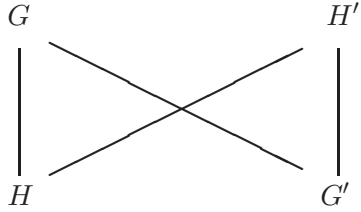
\[
\setlength{\unitlength}{0.8cm}
\begin{picture}(20,5)
\thicklines
\put(6.8,4){$G$}
\put(6.8,1){$H$}
\put(12.1,4){$H'$}
\put(12,1){$G'$}
\put(7,1.5){\line(0,1){2.1}}
\put(12.3,1.5){\line(0,1){2.1}}
\put(7.5,1.5){\line(2,1){4.2}}
\put(7.5,3.7){\line(2,-1){4.2}}
\end{picture}
\]
and the associated see-saw identity
\[
\langle \Theta_{G',G}(\pi_{G'}),\pi_H\rangle_{H\fq} =\langle \pi_{G'},\Theta_{H,H'}(\pi_H)\rangle_{G'\fq},
\]
where $\pi_H$ and $\pi_{G'}$ are irreducible representations of $H\fq$ and $G'\fq$ respectively.

First consider the case that
\[
G\cong\o^\epsilon_{2n},\ H\cong\o^{\epsilon'}_{2n-1}\times\o^{\epsilon''}_1,\ H'\cong\sp_{2n'} \times\sp_{2n'}\textrm{ and }G'\cong\sp_{2n'},
\]
 where  $\epsilon=\epsilon_{-1}\cdot \epsilon'\cdot\epsilon''$ so that $H$ is embedded into $G$  by \eqref{disc}, and $G'$ is embedded into $H'$ diagonally.   Then we have the see-saw diagram
\[
\setlength{\unitlength}{0.8cm}
\begin{picture}(20,5)
\thicklines
\put(6.6,4){$\o^\epsilon_{2n}$}
\put(5.6,1){$\o^{\epsilon'}_{2n-1}\times\o^{\epsilon''}_1$}
\put(11.2,4){$\sp_{2n'}\times\sp_{2n'}$}
\put(12,1){$\sp_{2n'}$}
\put(7,1.5){\line(0,1){2.1}}
\put(12.3,1.5){\line(0,1){2.1}}
\put(7.5,1.5){\line(2,1){4.2}}
\put(7.5,3.7){\line(2,-1){4.2}}
\end{picture}
\]
Similarly, consider the case that
 \[
 G\cong\o^\epsilon_{2n+1},\ H\cong\o^{\epsilon'}_{2n}\fq\times\o^{\epsilon''}_1,\ H'\cong\sp_{2n'}\times\sp_{2n'}\textrm{ and }G'\cong\sp_{2n'},
 \]
 where $\epsilon=\epsilon'\cdot\epsilon''$ so that $H$ is embedded into $G$ again by \eqref{disc}.
Then we have the see-saw diagram
\[
\setlength{\unitlength}{0.8cm}
\begin{picture}(20,5)
\thicklines
\put(6.1,4){$\o^\epsilon_{2n+1}$}
\put(5.9,1){$\o^{\epsilon'}_{2n}\times\o^{\epsilon''}_1$}
\put(11.2,4){$\sp_{2n'}\times\sp_{2n'}$}
\put(12,1){$\sp_{2n'}$}
\put(7,1.5){\line(0,1){2.1}}
\put(12.3,1.5){\line(0,1){2.1}}
\put(7.5,1.5){\line(2,1){4.2}}
\put(7.5,3.7){\line(2,-1){4.2}}
\end{picture}
\]

\section{First occurrence index for symplectic and orthogonal groups}\label{sec4}

The aim of this section is to prove the following result.
\begin{lemma} \label{4}
(i) Let $\pi$ be an irreducible cuspidal representation of $\o^{\epsilon}_{2m}\fq$ with $m\le k^2$. If $\pi$ is not unipotent, then
\[
n^{\epsilon}(\pi)<m+k;
\]
(ii) Let $\pi$ be an irreducible cuspidal representation of $\o_{2m+1}^{\epsilon}\fq$ with $m\le k(k-1)$. If $\pi$ is not unipotent, then
\[
n^{\epsilon}(\pi)<m+k.
\]
\end{lemma}

To prove this lemma, we need to determine the first occurrence indices of  cuspidal representations.
We begin  with reviewing Lusztig's results \cite{L1} on the  unipotent cuspidal representations of finite classical groups.

\begin{theorem} \label{fun} The following groups

(i) $\UU_n$, $n=k(k+1)/2$,

(ii) $\sp_{2n}$, $n=k(k+1)$,

(iii) $\rm{SO}_{2n+1}$, $n=k(k+1)$,

(iv) $\rm{SO}^{\epsilon}_{2n}$, $n=k^2$, $\epsilon=\epsilon(k)$,
\\are the only groups in their respective Lie families which possess a unipotent  cuspidal  representation. In each case, the specified group $G$ has a unique irreducible unipotent  cuspidal  representation.
\end{theorem}

Recall from the Introduction that $\pi_{\sp_{2k(k+1)}}$, $\pi_{\rm{SO}^\epsilon_{2k(k+1)+1}}$, $\epsilon=\pm$, and $\pi_{\rm{SO}^{\epsilon(k)}_{2k^2}}$ stand for the unique irreducible unipotent cuspidal representations of the corresponding groups.  When the family of classical groups is specified and no confusion can arise, we will simply denote these representations by $\pi_k$, and also denote by  $\pi^+_k$, $\pi^-_k$ the two irreducible unipotent cuspidal representations of $\o^\epsilon_{2k(k+1)+1}$, $\epsilon=\pm$, or $\o^{\epsilon(k)}_{2k^2}$.

The set $\E(G)$ of irreducible characters of $G^F$ can be partitioned
by geometric conjugacy classes (see e.g. \cite{L2})
\[
\mathcal{E}(G)=\coprod_{s}\mathcal{E}(G,(s)),
\]
where $s=s_{G^{*}} $ runs over the semisimple conjugacy classes of the dual group $ G^{*}$. By \cite{L3} there is a bijection
\[
\mathcal{L}:\mathcal{E}(G,(s)) \to \mathcal{E}(C_{G^{*}}(s),(1)).
\]
Moreover if the identity components of the centers of $G$ and $C_{G^*}(s)$ have the same $\Fq$-rank, then $\pi\in \mathcal{E}(G,(s))$ is cuspidal if and only if  $\mathcal{L}(\pi)\in \mathcal{E}(C_{G^*}(s),(1))$ is cuspidal (see e.g. \cite[Chap. 9]{L1}).  The $\theta$-representations of $G^F$ introduced in \cite{LW1} are those which appear in $\mathcal{E}(G,(\theta))$ for  certain distinguished quadratic elements $\theta$ in various maximal tori of $G^*$. Based on Lusztig's results, we obtained in {\it loc. cit.} the following classification of cuspidal $\theta$-representations.

\begin{corollary} \label{ctheta}
The following groups

(i) $\UU_n$, $n=k(k+1)/2$,

(ii) $\rm{SO}_{2n+1}$, $n=k(k+1)$,

(iii) $\rm{SO}^{\epsilon}_{2n}$, $n=k^2$, $\epsilon=\epsilon(k)$,

(iv) $\sp_{2n}$, $n=k^2$,
\\
are the only groups in their respective Lie families which possess cuspidal $\theta$-representations. In the first three cases, the specified group $G$ has a unique irreducible cuspidal $\theta$-representation $\pi^\theta_k$. For symplectic groups, each $\sp_{2k^2}$ has two irreducible cuspidal $\theta$-representations $\pi^\theta_{k,\alpha}$ and $\pi^\theta_{k,\beta}$.
\end{corollary}

%Similar to the unipotent case, we denote $\pi^{\theta,+}_k$, $\pi^{\theta, -}_k$ the two cuspidal $\theta$-representations of $\o_{2k(k+1)+1}$ or $\o^\epsilon_{2k^2}$, $\epsilon=\rm{sgn}(-1)^k$.
The initial representations $\pi_{\o_2^-}^\pm$ and $\pi_{1,i}^\theta$, $i=\alpha, \beta$ have been specified in the Introduction. The recipe for labeling the whole chain of cuspidal unipotent and $\theta$-representations via conservations are given by the following two theorems. 

\begin{theorem}[\cite{AM}, Theorem 5.2] \label{even}
The theta correspondence for dual pairs $(\sp_{2n}, \o^\epsilon_{2n'})$ takes unipotent cuspidal representations to unipotent cuspidal representations as follows:

(i)  $(\sp_{2k(k+ 1)}, \o^{\epsilon(k)} _{2k^2})$,
\[
\Theta^{\epsilon(k)}_{k(k+1), k^2}: \pi_{\sp_{2k(k+1)}}\longrightarrow  \pi^{-}_{\rm{O}^{\epsilon(k)}_{2k^2}};
\]

(ii)  $(\sp_{2k(k+ 1)}, \o^{\epsilon(k+1)} _{2(k+1)^2})$,
\[
\Theta^{\epsilon(k+1)}_{k(k+1), (k+1)^2}: \pi_{\sp_{2k(k+1)}}\longrightarrow  \pi^{+}_{\rm{O}^{\epsilon(k+1)}_{2(k+1)^2}}.
\]
\end{theorem}

We remark that for symplectic and even orthogonal dual pairs, the above result of theta correspondence for unipotent cuspidal representations does not depend on the choice of $\psi$.

\begin{theorem}[\cite{LW1}, Theorem 3.12]\label{spthetalift}
Let $\pi^\theta_{k,i}$, $i=\alpha,\beta$ be the irreducible cuspidal $\theta$-representations of $\sp_{2k^2}\fq$, and let $n^{\epsilon}(\pi^\theta_{k,i})$ be the first occurrence index of $\pi^\theta_{k,i}$ in the Witt tower ${\bf O}^\epsilon_\rm{odd}$. Then  one has $n^{+}(\pi^\theta_{k,\alpha})=n^{-}(\pi^\theta_{k,\beta})=k(k-1)$ and $n^{-}(\pi^\theta_{k,\alpha})=n^{+}(\pi^\theta_{k,\beta})=k(k+1)$. The theta correspondence is given by
\begin{align*}
&\Theta^{+}_{k^2, k(k-1)}: \pi^\theta_{k,\alpha}\longrightarrow \pi_{k-1}^{\epsilon(k)},\quad \Theta^{-}_{k^2, k(k+1)}:\pi^\theta_{k,\alpha}\longrightarrow \pi_{k}^{\epsilon(k)},\\
& \Theta^{+}_{k^2, k(k+1)}: \pi^\theta_{k,\beta}\longrightarrow\pi_{k}^{\epsilon(k)},\quad \Theta^{-}_{k^2, k(k-1)}: \pi^\theta_{k,\beta}\longrightarrow\pi_{k-1}^{\epsilon(k)},
\end{align*}
where $\pi^\pm_k$ denotes the irreducible unipotent cuspidal representations of $\rm{O}^\epsilon_{2k(k+1)+1}$ given by (\ref{un}).
\end{theorem}

Lemma \ref{4}  follows immediately from Theorem \ref{even},  Theorem \ref{spthetalift} and \cite[Theorem 6.9 and Theorem 7.9]{P3}.

\section{Bessel case of Theorem \ref{main}}\label{sec5}
In this section we study the branching of unipotent cuspidal representations of finite orthogonal groups. We
will prove the following result, which is the Bessel case of Theorem \ref{main}.
\begin{theorem}\label{be}
(i) For an irreducible unipotent cuspidal representation  $\pi^\eta_{\rm{O}^\epsilon_{2k(k+1)+1}}$ of $\rm{O}^\epsilon_{2k(k+1)+1}\fq$, one has $\ell_0^\rm{B}(\pi^\eta_{\rm{O}^\epsilon_{2k(k+1)+1}})=k$ and 
 \[
 \CD^\rm{B}_{k, v_0}(\pi^\eta_{\rm{O}^{\epsilon}_{2k(k+1)+1}})=\pi^{\eta \cdot \epsilon(k)}_{\rm{O}^{\epsilon(k)}_{2k^2}}, 
\]
where $Q(v_0)= \epsilon \cdot \epsilon(k)$.

(ii) For an irreducible unipotent cuspidal representation $\pi^\eta_{\rm{O}^{\epsilon(k)}_{2k^2}}$ of $\rm{O}^{\epsilon(k)}_{2k^2}\fq$,  one has 
$\ell_0^\rm{B}(\pi_{\rm{O}^{\epsilon(k)}_{2k^2}})=k-1$ and
\[
 \CD^\rm{B}_{k-1, v_0}(\pi^\eta_{\rm{O}^{\epsilon(k)}_{2k^2}})=\pi^{\eta\cdot \epsilon(k-1)}_{\rm{O}^{\epsilon}_{2k(k-1)+1}},
\]
where $Q(v_0)=\epsilon_{-1}\cdot \epsilon\cdot\epsilon(k)$.\end{theorem}

\subsection{Reduction to the basic case} We first show that the parabolic induction preserves multiplicities, and thereby make a reduction to the basic case.
 From \cite[Proposition 5.2]{LW1}, we know that the parabolic induction preserves multiplicities between unipotent representations of unitary groups. Namely,
\[
\langle \pi\otimes\bar{\nu}, \pi'\rangle _{H(\Fq)}=\langle I^{\UU_{n+1}}_{P}(\tau\otimes\pi'),\pi\rangle _{\UU_{n}(\Fq)}
\]
for irreducible unipotent representations $\pi$ and $\pi'$ of $\UU_n(\Fq)$ and $\UU_m(\Fq)$ respectively, where $P$ is an $F$-stable parabolic subgroup of $\UU_{n+1}$ with Levi factor $L^F\cong\GGL_\ell(\bb{F}_{q^2})\times \UU_{n+1-2\ell}\fq$, and $\tau$ is an irreducible cuspidal representation of $\GGL_\ell(\bb{F}_{q^2})$. In the same manner, we have the following analog for orthogonal groups with $\pi$ unipotent, which reduces the calculation to the basic case.

\begin{proposition}\label{so1}
Let $\pi$ be an irreducible unipotent representation of $\rm{SO}^\epsilon_n(\Fq)$, and $\pi'$ be an irreducible representation of $\rm{SO}^{\epsilon'}_m(\Fq)$ with $n > m$, $n\equiv m+1$ mod $2$. Let $P$ be an $F$-stable maximal parabolic subgroup of $\rm{SO}^{\epsilon'}_{n+1}$ with Levi factor $\GGL_{\ell} \times \rm{SO}^{\epsilon'}_{m}$, $\ell=(n+1-m)/2$, and $\tau$ be an irreducible cuspidal representation of $\GGL_{\ell}\fq $ which is nontrivial if $\ell=1$. Then we have
\begin{equation}\label{so11}
m(\pi, \pi')=\langle \pi\otimes \bar{\nu}, \pi'\rangle _{H(\Fq)}=\langle I^{\rm{SO}^\e_{n+1}}_{P}(\tau\otimes\pi'),\pi\rangle _{\rm{SO}^\epsilon_{n}(\Fq)},
\end{equation}
where the data $(H,\nu)$ is given by (\ref{hnu}).
\end{proposition}

\begin{proof}
It can be proved in the same way as \cite[Theorem 15.1]{GGP1}. The assumption of $\pi$ in \cite[Theorem 15.1]{GGP1} was used to obtain the following statement: for an $F$-stable maximal parabolic subgroup $P'$ of $\rm{SO}^\epsilon_{n}$ with Levi factor $\GGL_{(n+1-m)/2} \times \rm{SO}^\epsilon_{m-1}$ ,
\[
\langle I^{\rm{SO}^\epsilon_{n}}_{P'}\left(\tau\otimes(\pi'|_{\rm{SO}^\epsilon_{m-1}(\Fq)})\right),\pi\rangle _{\rm{SO}^\epsilon_{n}(\Fq)}=0.
\]
Since in our case $\pi$ is unipotent, this  multiplicity is nonzero
only if $\tau$ and $\pi'|_{\rm{SO}^\epsilon_{m-1}(\Fq)}$ are both unipotent. It is well-known that $\GGL_\ell(\bb{F}_{q})$ has no unipotent cuspidal representations if $\ell>1$. By the assumption on $\tau$,
 it is not unipotent. Therefore the above multiplicity is zero.
The rest of the proof is the same as that of \cite[Theorem 15.1]{GGP1}.
\end{proof}

For later use, we generalize Proposition \ref{so1} as follows.

\begin{proposition} \label{7.21}
Let $\pi$ be an irreducible unipotent representation of $\rm{SO}^\epsilon_n(\Fq)$, and $\pi'$ be a  representation of $\rm{SO}^{\epsilon'}_m(\Fq)$ with $n > m$, $n\equiv m+1$ mod $2$. Let $P$ be an $F$-stable maximal parabolic subgroup of $\rm{SO}^{\epsilon'}_{n+1}$ with Levi factor $\GGL_{\ell} \times \rm{SO}^{\epsilon'}_{m}$, $\ell=(n+1-m)/2$. Let $\tau_1$ (resp. $\tau_2$) be an irreducible cuspidal representations of $\GGL_{\ell'}(\bb{F}_{q})$ (resp. $\GGL_{\ell-\ell'}(\bb{F}_{q}))$, $\ell'\leq \ell$, which is nontrivial if $\ell'=1$ (resp.  $\ell-\ell'=1$), and
\[
\tau=I_{\GGL_{\ell'}\times  \GGL_{\ell- \ell'}}^{\GGL_\ell}(\tau_1\times\tau_2).
\]
Then we have
\[
m(\pi, \pi')=\langle \pi\otimes \bar{\nu}, \pi'\rangle _{H(\Fq)}=\langle I^{\so^\e_{n+1}}_{P}(\tau\otimes\pi'),\pi\rangle _{\so^\epsilon_{n}(\Fq)},
\]
where the data $(H,\nu)$ is given by (\ref{hnu}).
\end{proposition}

\begin{proof}
It can be proved in the same way as \cite[Theorem 15.1]{GGP1}, where it was established for non-archimedean local fields, and the proof works for finite fields as well. We follow the notations in \cite{GGP1}.  Let $V$ be an $n$-dimensional space over $\Fq$ with a non-degenerate symmetric bilinear form $(,)$, which defines the special orthogonal group $\so(V)=\so^\epsilon_n\fq$ and $W \subset V$ be an $m$-dimensional non-degenerate orthogonal subspace, so that
\[
W^\bot=X+X^\vee+E
\]
where $E = \bb{F}_{q} \cdot e$ is an anisotropic line and
\[
X=\langle v_1,\ldots,v_{\ell-1}\rangle
\]
is an isotropic subspace
with $\dim X = \ell-1$ and $X^\vee$ is the dual of $X$. Let
\[
E^- = \bb{F}_{q} \cdot f
\]
denote the rank 1 space equipped with a form which is the negative of that on $E$, so that $E + E^-$ is a split rank 2 space. The two isotropic lines in $E + E^-$  are spanned by
\[
v = e + f \quad \textrm{and}\quad v' =\frac{1}{2(e,e)} (e - f).
\]
Now consider the space
\[
W' = V \oplus E^-
\]
which contains $V$ with codimension 1 and isotropic subspaces
\[
Y = X +\bb{F}_q\cdot v\quad \textrm{and}\quad Y ^\vee= X^\vee +\bb{F}_q\cdot v'.
\]
Hence we have
\[
W' = Y + Y^\vee + W.
\]
Let $P = P (Y )$ be the parabolic subgroup of $\so(W')$ stabilizing $Y $ and let $M(Y)$ be its Levi subgroup stabilizing $Y$ and $Y^\vee$. Then $\so(W')=\so^\e_{n+1}\fq$ and $M(Y)=\GGL_\ell\fq \times \so^\e_m\fq$. Let $P_V (X )$ be the parabolic subgroup of $\so(V)$ stabilizing $X$, so that
\[
P_V(X) = M_V(X) \ltimes N_V(X)
\]
where $N_V(X)$ is the unipotent radical of $P_V(X)$. Let $Q$ be a subgroup of $P_V (X)$ given by
\[
Q=(\GGL(X)\times \so(W))\ltimes N_V(X).
\]
As in the proof \cite[Theorem 15.1]{GGP1}, one has the following commutative diagram with exact rows
\[
\xymatrix
{
0  \ar[r] & N(Y)  \ar[r]  & P(Y)  \ar[r]  & \GGL(Y)\times \so(W) \ar[r] & 0\\
0  \ar[r] & N(Y)\cap Q  \ar[r] \ar[u]  &Q  \ar[r] \ar[u] & R\times \so(W) \ar[u] \ar[r] & 0
}
\]
where
$
R\subset \GGL(Y)
$
is the mirabolic subgroup which stabilizes the subspace $X \subset Y$ and fixes $v$ modulo $X$. Note also that $N(Y ) \cap Q \subset N_V(X)$ and
\[
N_V (X)/(N(Y) \cap Q)  \cong \rm{Hom}(E, X).
\]
As a consequence, one has
\[
(\tau\otimes\pi')|_Q=\tau|_R\otimes \pi'.
\]
By the proof of \cite[Theorem 15.1]{GGP1}, it suffices to show that
\[
\langle\pi,\rm{Ind}^{\so(V)}_{Q}(\tau|_R\otimes \pi')\rangle_{\so(V)}=\langle\pi,\rm{Ind}^{\so(V)}_{Q}(\rm{Ind}^R_U\chi\otimes \pi')\rangle_{\so(V)}
\]
where $U$ is the unipotent radical of the Borel subgroup of $\GGL(Y )$ stabilizing the flag
\[
\langle v_1\rangle\subset\langle v_1,v_2\rangle\subset\cdots\subset\langle v_1,\ldots,v_{\ell-1},v\rangle=Y,
\]
and $\chi$ is any generic character of $U$.

Let $N_n$ be the group of upper triangular unipotent matrices in $\GGL_n(\bb{F}_{q})$. Recall that  $\psi$ is a fixed nontrivial additive character of \
$\bb{F}_{q}$. Let $\psi_n$ be the character of $N_n$ given by
\[
\psi_n(u) = \psi(u_{1,2} + u_{2,3} + \ldots + u_{n-1,n}).
\]
Let $R^n_i$ be the subgroup of $\GGL_n(\bb{F}_{q})$ consisting of matrices
\[
 \left(\begin{matrix}
   g      & v  \\
   0      & z
  \end{matrix}\right)
 \]
 with $g\in \GGL_{i}(\bb{F}_{q})$, $v\in M_{i\times (n-i)}\fq$, $z\in N_{n-i}\fq$, so that $R^n_i=\GGL_{i}(\bb{F}_{q})\ltimes V_{n-i}$, where $V_{n-i}$ is the unipotent radical of $R^n_i$.

By the theory of Bernstein-Zelevinsky derivatives (c.f. \cite[Corollary 4.3]{GGP2}),
\[
\tau|_R=\rm{Ind}^R_U\chi+\rm{Ind}^R_{R^\ell_{\ell'}}\tau_1\otimes\psi_{\ell-\ell'}+\rm{Ind}^R_{R^\ell_{\ell-\ell'}}\tau_2\otimes\psi_{\ell'}.
\]
Let $Q'$ be the subgroup of $Q$ given by
\[
Q'=(R^{\ell-1}_{\ell'}\times \so(W))\ltimes (N(Y) \cap Q).
\]
Then there is an $F$-stable maximal parabolic subgroup $P_{\ell'}$ of $\so^\epsilon_{n}$ with Levi factor $\GGL_{\ell'} \times \so^\epsilon_{n-2\ell'}$ such that $Q'\subset P_{\ell'}$. Thus we get
\[
\begin{aligned}
& \langle\pi,\rm{Ind}^{\so(V)}_{Q}(\rm{Ind}^R_{R_{\ell'}}\tau_1\otimes\psi_{\ell-\ell'}\otimes \pi')\rangle_{\so(V)}\\
=&\langle\pi,\rm{Ind}^{\so(V)}_{Q'}(\tau_1\otimes\psi_{\ell-\ell'}\otimes \pi')\rangle_{\so(V)}\\
=&\langle\pi,I^{\so(V)}_{P_{\ell'}}(\tau_1\otimes\rm{Ind}^{\so^\epsilon_{n-2\ell'}}_{\so^\epsilon_{n-2\ell'}\cap Q'}(\psi_{\ell-\ell'}\otimes \pi'))\rangle_{\so(V)}.
\end{aligned}
\]
By our assumption, $\pi$ is unipotent and $\tau_1$ is not, hence
\[
\langle\pi,I^{\so(V)}_{P_{\ell'}}(\tau_1\otimes\rm{Ind}^{\so^\epsilon_{n-2\ell'}}_{\so^\epsilon_{n-2\ell'}\cap Q'}(\psi_{\ell-\ell'}\otimes \pi'))\rangle_{\so(V)}=0.
\]
In the same manner, one has
\[
\langle\pi,\rm{Ind}^{\so(V)}_{Q}(\rm{Ind}^R_{R^\ell_{\ell-\ell'}}\tau_2\otimes\psi_{\ell'}\otimes \pi')\rangle_{\so(V)}=0.
\]
It follows that
\[
\begin{aligned}
& \langle\pi,\rm{Ind}^{\so(V)}_{Q}\tau|_R\otimes \pi'\rangle_{\so(V)}\\
=&\langle\pi,\rm{Ind}^{\so(V)}_{Q}(\rm{Ind}^R_U\chi+\rm{Ind}^R_{R^\ell_{\ell'}}\tau_1\otimes\psi_{\ell-\ell'}+\rm{Ind}^R_{R^\ell_{\ell-\ell'}}\tau_2\otimes\psi_{\ell'})\otimes \pi'\rangle_{\so(V)}\\
=&\langle\pi,\rm{Ind}^{\so(V)}_{Q}(\rm{Ind}^R_U\chi\otimes \pi')\rangle_{\so(V)},
\end{aligned}
\]
which completes the proof.
\end{proof}

\begin{corollary}\label{o4}
Keep the assumptions in Proposition \ref{7.21}. Then
\begin{equation} \label{eq-cor}
m(\pi, \pi')=\langle I^{\so^\e_{n+1}}_{P}(\tau\otimes\pi'),\pi\rangle _{\so^\epsilon_{n}(\Fq)}=m\left(I^{\so^\e_{n+1-2\ell'}}_{\GGL_{\ell-\ell'}\times \so^\e_{m}}(\tau_2\otimes\pi'),\pi\right).
\end{equation}
\end{corollary}
\begin{proof}
By Proposition \ref{so1} and Proposition \ref{7.21}, we have
\begin{align*}
m(\pi, \pi')=&\langle I^{\so^\e_{n+1}}_{P}(\tau\otimes\pi'),\pi\rangle _{\so^\epsilon_{n}(\Fq)}\\
=&\langle I^{\so^\e_{n+1}}_{P}((I^{\GGL_{\ell}}_{\GGL_{\ell'}\times \GGL_{\ell-\ell'}}(\tau_1\otimes\tau_2))\otimes\pi'),\pi\rangle _{\so^\epsilon_{n}(\Fq)}\\
=&\langle I^{\so^\e_{n+1}}_{\GGL_{\ell'}\times\so^\e_{n-1}}(\tau_1\otimes(I^{\so^\e_{n+1-2\ell'}}_{ \GGL_{\ell-1}\times \so^\e_{m}}(\tau_2\otimes\pi')),\pi\rangle _{\so^\epsilon_{n}(\Fq)}\\
=&m\left( I^{\so^\e_{n+1-2\ell'}}_{ \GGL_{\ell-\ell'}\times \so^\e_{m}}(\tau_2\otimes\pi'),\pi\right).
\end{align*}
\end{proof}

\subsection{Reformulation} To prove Theorem \ref{be}, by Proposition \ref{so1} and Corollary \ref{o4} it suffices to calculate (\ref{so11}) or (\ref{eq-cor}). In the rest of this section, we will take $\ell'=1$ in Proposition \ref{7.21} and Corollary \ref{o4}; and in order to apply the theta correspondence we will work with orthogonal groups instead of special orthogonal groups. It is not hard to see that Theorem \ref{be} readily follows from  Theorem \ref{even}, Theorem \ref{spthetalift}, and Theorem \ref{o1} below whose formulation is more adaptable for making induction argument. 

\begin{theorem}\label{o1}

(i) Let $\pi=\pi^{\eta}_{k}$ be an irreducible unipotent cuspidal representation of $\rm{O}_{2k(k+1)+1}^\epsilon\fq$, and $\pi'$ be an irreducible representation of $\o^\e_{2m}\fq$. Then the following hold.
\begin{itemize}
\item
If $m<k^2$, then
$
m(\pi,\pi')=0.
$
\item
If $m=k^2$, then
\[
m(\pi,\pi')=\left\{
\begin{array}{ll}
1, &  \textrm{if }\epsilon'=\epsilon(k)\textrm{ and } \pi'=\pi^{\eta'}_k,\\
0, & \textrm{otherwise,}
\end{array}\right.
\]
where $\pi^{\eta'}_k$ is the irreducible unipotent cuspidal representation of $\rm{O}^{\epsilon(k)}_{2k^2}\fq$ such that
\[
(n^\epsilon(\pi)-k(k+1))(n^{\epsilon(k)}(\pi^{\eta'}_k)-k^2)>0.
\]
\end{itemize}
(ii) Let $\pi=\pi_k^\eta$ be an irreducible unipotent cuspidal representation of $\rm{O}^{\epsilon(k)}_{2k^2}\fq$, and $\pi'$ be an irreducible representation of $\o_{2m+1}^\e\fq$. Then the following hold.
\begin{itemize}
\item
If $m<k(k-1)$, then
$
m(\pi,\pi')=0.
$
\item If $m=k(k-1)$, then
\[
m(\pi,\pi')=\left\{
\begin{array}{ll}
1, &  \textrm{if }\ \pi'=\pi^{\eta'}_{k-1},\\
0, & \textrm{otherwise,}
\end{array}\right.
\]
where $\pi^{\eta'}_{k-1}$ is the irreducible unipotent cuspidal representation of $\rm{O}_{2k(k-1)+1}^\e\fq$ such that
\[
(n^{\epsilon(k)}(\pi)-k^2)(n^\e(\pi^{\eta'}_{k-1})-k(k-1))>0.
\]
\end{itemize}
\end{theorem}

The rest of this section is devoted to the proof of Theorem \ref{o1}, which will be divided into two parts.

\subsection{Vanishing result} As the first step towards the proof, we establish the cases where the multiplicity in Theorem \ref{o1} vanishes.

\begin{proposition}\label{o2}
(i) Let $\pi$ and $\pi'$ be irreducible representations of $\o_{2n+1}^\epsilon\fq$ and $\o^{\epsilon'}_{2m}\fq$ respectively, $m\leq n$. Let $\sigma$ and $\sigma'$ be irreducible cuspidal representations of $\o_{2n^*+1}^\epsilon\fq$ and $\o^{\epsilon'}_{2m^*}\fq$, respectively,  $n^*\le n$, $m^*\le m$, such that  $\pi\in \cal{E}(\o_{2n+1}^\epsilon,\sigma)$ and $\pi'\in\cal{E}(\o^\e_{2m},\sigma')$. Let $\ell=n+1-m$ and $\tau$ be an irreducible cuspidal representation of $\GGL_{\ell}\fq$, nontrivial if $\ell=1$. If one of the following holds:
\begin{itemize}
\item[]
Case (A): $n^{\epsilon}(\sigma)-n^*>0$ and $n^{\epsilon}(\sigma)-n^*-1>n^{\epsilon'}(\sigma')-m^*$;
\item[]
Case (B): $n^{\epsilon}(\sigma)-n^*<0$ and $n^{\epsilon}(\sigma)-n^*<n^{\epsilon'}(\sigma')-m^*$,
\end{itemize}
 then
\[
\langle I^{\o^{\epsilon'}_{2n+2}}_{\GGL_{n+1-m}\times\o^\e_{2m}}(\tau\otimes\pi'),\pi\rangle _{\o^{\epsilon}_{2n+1}\fq}=0.
\]
(ii) Let $\pi$ and $\pi'$ be irreducible representations of $\o^\epsilon_{2n}\fq$ and $\o^{\epsilon'}_{2m+1}\fq$ respectively, $m<n$. Let $\sigma$ and $\sigma'$ be irreducible cuspidal representations of $\o^\epsilon_{2n^*}\fq$ and $\o^{\epsilon'}_{2m^*+1}\fq$ respectively,  $n^*\le n$, $m^*\le m$, such that $\pi\in \cal{E}(\o^\epsilon_{2n},\sigma)$ and $\pi'\in \cal{E}(\o_{2m+1},\sigma')$. Let $\ell=n-m$ and $\tau$ be an irreducible cuspidal representation of $\GGL_{\ell}\fq$, nontrivial if $\ell=1$. If one of the following holds:
\begin{itemize}
\item[]
Case (A): $n^{\epsilon}(\sigma)-n^*>0$ and $n^{\epsilon}(\sigma)-n^*>n^{\epsilon'}(\sigma')-m^*$;
\item[]
Case (B): $n^{\epsilon}(\sigma)-n^*<0$ and $n^{\epsilon}(\sigma)-n^*+1<n^{\epsilon'}(\sigma')-m^*$,
\end{itemize}
 then
\[
\langle I^{\o^\e_{2n+1}}_{\GGL_{n-m}\times \o^\e_{2m+1}}(\tau\otimes\pi'),\pi\rangle _{\o^{\epsilon}_{2n}\fq}=0.
\]
\end{proposition}

We will prove the above proposition by the standard arguments of theta correspondence and see-saw dual pairs. To this end we first need to know the theta correspondence of representations in the Harish-Chandra series $\cal{E}(G,\sigma)$ for a cuspidal representation $\sigma$.

\begin{proposition} \label{o3}
Let  $(G_m, G'_{m'})$ be a dual pair  in the Witt tower  $(\bf{Sp},\bf{O}^\epsilon_{\rm{even}})$ or $(\bf{Sp},\bf{O}^\epsilon_{\rm{odd}})$. Assume that $\pi\in \cal{E}(G_m,\sigma)$, where $\sigma$ is an irreducible cuspidal representation of $G_n^F$, $n\leq m$, $n\equiv m \ \rm{mod} \ 2$. Let $n' = n^\epsilon(\sigma)$ be its first occurrence index, so that $\sigma' := \Theta^\epsilon_{n,n'}(\sigma)$ is an irreducible cuspidal representation of $G_{n'}^{\prime F}$. Then the following hold.

(i) The irreducible constituents of $\Theta^\epsilon_{m,m'}(\pi)$ belong to   $\cal{E}(G_{m'}^{\prime },\sigma')$,

(ii) If $m'-m\ge {n'-n}$, then $\Theta^\epsilon_{m,m'}(\pi)\ne 0$.
\end{proposition}

\begin{proof}
We will only prove the proposition for $(G_n,G_{n'}')\in(\bf{Sp},\bf{O}^\epsilon_{\rm{odd}})$. The proof for symplectic and even orthogonal dual pairs is similar and will be left to
the reader.

We first prove (i) by induction on $m$.

 $\bullet$ Suppose that $ m = n$, i.e. $\pi=\sigma$ is cuspidal. Since $\Theta^\epsilon_{n,m'}(\pi)=0$ if $m'<n'$ and $\Theta^\epsilon_{n,n'}(\sigma)=\sigma'$, we may assume that $m'>n'$. It is known that (cf. \cite[Chap. 3]{MVW}) each constituent $\pi'$ of $\Theta^\epsilon_{n,m'}(\pi)$  is noncuspidal. Let $j$ be the positive integer such that $\pi'\subset I^{\o_{2m'+1}^\epsilon}_{\o^\epsilon_{2(m'-j)+1}\times \GGL_{j}}(\sigma_{1}'\otimes \rho')$ with $\sigma_{1}'\in  \E(\o^\epsilon_{2(m'-j)+1})$ cuspidal and $\rho'\in \E(\GGL_{j})$. Since $\pi'\subset\Theta^\epsilon_{n, m'}(\pi)$, one has
\[
\begin{aligned}
0&<\langle\omega^\epsilon_{n,m'},\pi \otimes \pi '\rangle_{\sp_{2n}\fq\times \o_{2m'+1}^\epsilon\fq} \\
&\leq \langle\omega^\epsilon_{n,m'},\pi \otimes I^{\o^\epsilon_{2m'+1}}_{\o^\epsilon_{2(m'-j)+1}\times \GGL_{j}}(\sigma_{1}'\otimes \rho')\rangle_{\sp_{2n}\fq\times \o^\epsilon_{2m'+1}\fq}
\\&=\langle J^{\o^\epsilon_{2m'+1}}_{\o^\epsilon_{2(m'-j)+1}\times \GGL_{j}}(\omega^\epsilon_{n,m'}),\pi \otimes \sigma_{1}'\otimes \rho'\rangle_{\sp_{2n}\fq\times \o^\epsilon_{2(m'-j)+1}\fq\times \GGL_{j}\fq}.
\end{aligned}
\]
Here $J$ standards for the Jacquet functor, which is adjoint to the induction functor $I$. We have the following decomposition (cf. \cite[Chap. 3, IV th.5]{MVW})
\[
\begin{aligned}
&J^{\o^\epsilon_{2m'+1}}_{\o^\epsilon_{2(m'-j)+1}\times \GGL_{j}}(\omega^\epsilon_{n,m'})
\\=&\bigoplus^{\min(n,j)}_{i=0}I^{\sp_{2n}\times \GGL_{j} \times \o^\epsilon_{2(m'-j)+1}}_{\sp_{2(n-i)}\times \GGL_{i} \times ( \GGL_{j-i}\times \GGL_{i})\times \o^\epsilon_{2(m'-j)+1}}(\omega^\epsilon_{n-i,m'-j} \otimes 1_{\GGL_{j-i}} \otimes \chi_{\GGL_{i}}R^{\GGL_{i}}).
\end{aligned}
\]
where $R^{\GGL_{i}}$ is the regular representation of $\GGL_{i}\fq$.
Hence $\langle\omega^\epsilon_{n,m'},\pi \otimes \pi '\rangle $ is bounded by
\[
\sum^{\min(n,j)}_{i=0}\langle\omega^\epsilon_{n-i,m'-j} \otimes 1_{\GGL_{j-i}} \otimes \chi_{\GGL_{i}}R^{\GGL_{i}},J^{\sp_{2n}\times \GGL_{j} \times \o^\epsilon_{2(m'-j)+1}}_{\sp_{2(n-i)}\times \GGL_{i} \times ( \GGL_{j-i}\times \GGL_{i})\times \o^\epsilon_{2(m'-j)+1}}(\pi \otimes \sigma_{1}'\otimes \rho')\rangle,
\]
where the scalar product in the $i$th summand is taken over the group
\[
\sp_{2(n-i)}\fq\times \o_{2(m'-j)+1}\fq\times \GGL_{j-i}\fq\times \GGL_i\fq\times \GGL_i\fq.
\]
Since $\pi=\sigma$ is cuspidal, the only nonzero term corresponds to $i = 0$, which implies that
\[
\langle\omega^\epsilon_{n,m'-j}\otimes 1_{\GGL_{j}},\pi \otimes \sigma_{1}'\otimes \rho'\rangle >0.
\]
It follows that $\rho'=1_{\GGL_{j}}$ and $\sigma_{1}'\subset\Theta^\epsilon_{n, m'-j}(\pi)$. Because $\sigma_{1}'$ is cuspidal,
we must have $m'-j=n'$ and $\sigma_1'=\sigma'$, i.e.  $\pi' \in \cal{E}(G_{m'}^{\prime F},\sigma')$.

$\bullet$ Suppose that $ m > n$. Then there exists $\rho \in \E(\GGL_{m-n})$ such that $\pi\subset I^{\sp_{2m}}_{\sp_{2n}\times \GGL_{n-m}}(\sigma\otimes \rho)$. For  an irreducible constituent $\pi'$ of  $\Theta^\epsilon_{m,m'}(\pi)$, we have
\begin{align*}
0&<\langle\omega^\epsilon_{m,m'},\pi \otimes \pi '\rangle_{\sp_{2m}\fq\times \o^\epsilon_{2m'+1}\fq}
\\&\leq \langle\omega^\epsilon_{m,m'},I^{\sp_{2m}}_{\sp_{2n}\times \GGL_{n-m}}(\sigma\otimes \rho) \otimes \pi '\rangle_{\sp_{2m}\fq \times \o^\epsilon_{2m'+1}\fq }
\\&=\langle J^{\sp_{2m}}_{\sp_{2n}\times \GGL_{n-m}}(\omega^\epsilon_{m,m'}),\sigma\otimes \rho \otimes\pi '\rangle_{ \sp_{2n}\fq \times \GGL_{m-n}\fq\times \o^\epsilon_{2m'+1}\fq }
\end{align*}
We have the decomposition
\begin{align*}
&J^{\sp_{2m}}_{\sp_{2n}\times \GGL_{m-n}}(\omega^\epsilon_{m,m'})
\\=&\bigoplus^{\min(m',m-n)}_{i=0}I^{\sp_{2n}\times \GGL_{m-n} \times \o^\epsilon_{2m'+1}}_{\sp_{2n}\times (  \GGL_{m-n-i}\times \GGL_{i})\times \GGL_{i}\times \o^\epsilon_{2(m'-i)+1}}(\omega^\epsilon_{n,m'-i} \otimes \chi_{\GGL_{m-n-i}} \otimes \chi_{\GGL_{i}}R^{\GGL_{i}}).
\end{align*}
Hence $\langle\omega^\epsilon_{m,m'},\pi \otimes \pi '\rangle $ is bounded by
\[
\sum^{\min(m',m-n)}_{i=0}\langle\omega^\epsilon_{n,m'-i} \otimes \chi_{\GGL_{m-n-i}} \otimes \chi_{\GGL_{i}}R^{\GGL_{i}},J^{\sp_{2n}\times \GGL_{m-n} \times \o^\epsilon_{2m'+1}}_{\sp_{2n}\times (  \GGL_{m-n-i}\times \GGL_{i})\times \GGL_{i}\times \o^\epsilon_{2(m'-i)+1}}(\sigma\otimes \rho \otimes\pi ')\rangle.
\]
By induction hypothesis and our result for $m=n$, if $\pi'\notin \cal{E}(G_{m'}',\sigma')$, then the above summation is zero,
which yields a contradiction.

To prove (ii), note that in this case $m'-m+n\ge n'$ and $\Theta^\epsilon_{n,m'-m+n}(\sigma)\ne 0$. It follows that there exists $\tau\in \E(\GGL_{{m'-m+n-n'}})$ such that
\[
\langle\Theta^\epsilon_{n,m'-m+n}(\sigma), I^{\o^\epsilon_{2(m'-m+n)+1}}_{\o^\epsilon_{2n'+1}\times\GGL_{m'-m+n-n'} } (\sigma'\otimes\tau)\rangle> 0.
\]
Then the required assertion follows from
\begin{align*}
&\langle\omega^\epsilon_{m,m'},\pi \otimes I^{\o^\epsilon_{2m'+1}}_{\o^\epsilon_{2(m'-m+n)+1}\times \GGL_{m-n}}(I^{\o^\epsilon_{2(m'-m+n)+1}}_{\o^\epsilon_{2n'+1}\times \GGL_{m'-m+n-n'}}(\sigma'\otimes\tau)\otimes \chi_{\GGL_{m-n}}\rho)\rangle_{\sp_{2m}\fq\times \o^\epsilon_{2m'+1}\fq}\\
=&\langle J^{\o_{2m'+1}}_{\o^\epsilon_{2(m'-m+n)+1}\times \GGL_{m-n}}(\omega^\epsilon_{m,m'}),\pi \otimes I^{\o^\epsilon_{2(m'-m+n)+1}}_{\o^\epsilon_{2n'+1}\times \GGL_{m'-m+n-n'}}(\sigma'\otimes\tau)\otimes \chi_{\GGL_{m-n}}\rho\rangle\\
=&\sum^{\min(m,m-n)}_{i=0}\langle\omega^\epsilon_{m-i,m'-m+n} \otimes \chi_{\GGL_{m-n-i}} \otimes \chi_{\GGL_{i}}R^{\GGL_{i}},
J^{\sp_{2m}\times \GGL_{m-n} \times \o^\epsilon_{2(m'-m+n)+1}}_{\sp_{2(m-i)}\times \GGL_{i} \times
( \GGL_{m-n-i}\times \GGL_{i})\times \o^\epsilon_{2(m'-m+n)+1}}\\
&\quad (\pi \otimes \chi_{\GGL_{m-n}}\rho\otimes I^{\o^\epsilon_{2(m'-m+n)+1}}_{\o^\epsilon_{2n'+1}\times \GGL_{m'-m+n-n'}}(\sigma'\otimes\tau) )\rangle\\
\ge & \langle\omega^\epsilon_{n,m'-m+n} \otimes \chi_{\GGL_{m-n}}R^{\GGL_{m-n}},
J^{\sp_{2m}\times \GGL_{m-n} \times \o^\epsilon_{2(m'-m+n)+1}}_{\sp_{2n}\times \GGL_{m-n} \times \GGL_{m-n}\times \o^\epsilon_{2(m'-m+n)+1}}\\
 & \quad (\pi \otimes\chi_{\GGL_{m-n}}\rho\otimes I^{\o^\epsilon_{2(m'-m+n)+1}}_{\o^\epsilon_{2n'+1}\times \GGL_{m'-m+n-n'}}(\sigma'\otimes\tau))\rangle\\
\ge& \langle\omega^\epsilon_{n,m'-m+n} \otimes \chi_{\GGL_{m-n}}R^{\GGL_{m-n}},
(\sigma\otimes \rho)\otimes\chi_{\GGL_{m-n}}\rho\otimes I^{\o^\epsilon_{2(m'-m+n)+1}}_{\o^\epsilon_{2n'+1}\times \GGL_{m'-m+n-n'}}(\sigma'\otimes\tau))\rangle\\
>&0.
\end{align*}
\end{proof}

\begin{proof} (of Proposition \ref{o2})
We only prove (i). The proof of (ii) is similar and will be left to the reader.  To ease notations we suppress various Levi subgroups from
the parabolic induction in the sequel, which should be clear from the context.

Note that if $n^{\epsilon}(\sigma)>n^*$ and $n^{\epsilon}(\sigma)-n^*-1>n^{\epsilon'}(\sigma')-m^*$, then by the conservation relation for cuspidal representations given in \cite[Theorem 12.3]{P1}, one has
\[
n^{\epsilon}(\sigma \otimes\rm{sgn})< n^*\quad \rm{and} \quad n^{\epsilon}(\sigma \otimes\rm{sgn})-n^*<n^{\epsilon'}(\sigma' \otimes\rm{sgn})-m^*.
\]
 On the other hand it is clear that
\begin{align*}
& \pi\in \cal{E}(\o^\epsilon_{2n},\sigma) \Longleftrightarrow \pi \otimes\rm{sgn}\in \cal{E}(\o^\epsilon_{2n},\sigma \otimes\rm{sgn}),\\
& \pi'\in  \cal{E}(\o^\e_{2m},\sigma')\Longleftrightarrow \pi' \otimes\rm{sgn}\in \cal{E}(\o^\e_{2m},\sigma' \otimes\rm{sgn})
\end{align*}
and
\[
\langle I^{\o^{\epsilon'}_{2n+2}}(\tau\otimes\pi'),\pi\rangle _{\o^{\epsilon}_{2n+1}\fq}=\langle I^{\o^{\epsilon'}_{2n+2}}(\tau\otimes(\pi' \otimes\rm{sgn})),\pi \otimes\rm{sgn}\rangle _{\o^{\epsilon}_{2n+1}\fq}.
\]
Hence it suffices to prove Case (A).

Put \[
n_+=n+1+n^{\epsilon'}(\sigma')-m^* \quad \textrm{ and }\quad n_-=n_+-(n^{\epsilon}(\sigma)-n^*-1).
\]
By our assumption, $n_-<n+1$.
Consider the see-saw diagram
\[
\setlength{\unitlength}{0.8cm}
\begin{picture}(20,5)
\thicklines
\put(6.3,4){$\sp_{2n_+}\times \sp_{2n_+}$}
\put(7,1){$\sp_{2n_+}$}
\put(12.2,4){$\o^{\epsilon'}_{2(n+1)}$}
\put(11.5,1){$\o^\epsilon_{2n+1}\times \o_1^{\epsilon''}$}
\put(7.7,1.5){\line(0,1){2.1}}
\put(12.8,1.5){\line(0,1){2.1}}
\put(8,1.5){\line(2,1){4.2}}
\put(8,3.7){\line(2,-1){4.2}}
\end{picture}
\]
where $\epsilon''=\epsilon_{-1}\cdot \epsilon\cdot\epsilon'$.

By Proposition \ref{o3} (ii) and Proposition \ref{w1}, for any irreducible $\rho'\subset I^{\o^{\epsilon'}_{2(n+1)}}(\tau\otimes\pi')$, there exists $\pi_1\in \E(\sp_{2(n_+-\ell)},\Theta^{\epsilon'}_{m^*,n^{\epsilon'}(\sigma')}(\sigma'))$ and irreducible $\rho_1\subset I^{\sp_{2n_+}}(\tau\otimes\pi_1)$ such that
\[
\rho'\subset \Theta^{\epsilon'}_{n_+,n+1}(\rho_1).
\]
Then we have
\[
\begin{aligned}
\langle \pi,\rho'\rangle_{\o^\epsilon_{2n+1}\fq}&\le
\langle \pi,\Theta^{\epsilon'}_{n_+,n+1}(\rho_1)
\rangle_{\o^\epsilon_{2n+1}\fq}=\langle \Theta^\epsilon_{n,n_+}(\pi)\otimes \omega^{\epsilon''}_{n_+},\rho_1\rangle_{\sp_{2n_+}\fq }.\\
\end{aligned}
\]
If $\Theta^\epsilon_{n,n_+}(\pi)=0$ then $\langle \pi',I^{\o^{\epsilon'}_{2(n+1)}}(\tau\otimes\pi)\rangle_{\o^\epsilon_{2n+1}}=0$. In particular, if $\pi=\sigma$, then by our assumption,
\[
n_+=n+1+n^{\epsilon'}(\sigma')-m^*<n+n^{\epsilon}(\sigma)-n^*=n^{\epsilon}(\sigma).
\]
It follows that $\Theta^\epsilon_{n,n_+}(\sigma)=0$ and the multiplicity is $0$. If $\Theta^\epsilon_{n,n_+}(\pi)\ne0$, then by our assumption and Proposition \ref{o3} (i), for any irreducible $\pi_{\sp}\subset\Theta^\epsilon_{n,n_+}(\pi)$, we have
\[
\pi_{\sp}\in\cal{E}(\sp_{2n_+},\Theta^\epsilon_{m,n^\epsilon(\sigma)}(\sigma)).
\]
Applying Proposition \ref{o3} and Proposition \ref{w1} again, for any such $\pi_{\sp}$, there exists
\begin{equation}\label{pi2}
\pi_2\in \cal{E}(\o^\epsilon_{2(n_+-(n^{\epsilon}(\sigma)-n^*))+1},\sigma)=\cal{E}(\o^\epsilon_{2(n_--1)+1},\sigma)
\end{equation}
such that
\[
\pi_{\sp}\subset \Theta^\epsilon_{n_--1,n_+} (\pi_2).
\]

Consider the see-saw diagram
\[
\setlength{\unitlength}{0.8cm}
\begin{picture}(20,5)
\thicklines
\put(6.3,4){$\sp_{2n_+}\times \sp_{2n_+}$}
\put(7.3,1){$\sp_{2n_+}$}
\put(12.4,4){$\o^{\epsilon'}_{2n_-}$}
\put(11.4,1){$\o^\epsilon_{2n_--1}\times \o_1^{\epsilon''}$}
\put(7.7,1.5){\line(0,1){2.1}}
\put(12.8,1.5){\line(0,1){2.1}}
\put(8,1.5){\line(2,1){4.2}}
\put(8,3.7){\line(2,-1){4.2}}
\end{picture}
\]
where $\epsilon'=\epsilon_{-1}\cdot \epsilon\cdot\epsilon''$.
One has
\[
\langle \pi_{\sp}\otimes \omega^{\epsilon''}_{n_+},I^{\sp_{2n_+}}(\tau\otimes\pi_1)\rangle_{\sp_{2n_+}\fq}
\leq \langle \Theta^\epsilon_{n_--1,n_+} (\pi_2)\otimes \omega^{\epsilon''}_{n_+},I^{\sp_{2n_+}}(\tau\otimes\pi_1)\rangle_{\sp_{2n_+}\fq}.
\]
For any irreducible  $\rho_2\subset I^{\sp_{2n_+}}(\tau\otimes\pi_1)$,
\[
\langle \Theta^\epsilon_{n_--1,n_+} (\pi_2)\otimes \omega^{\epsilon''}_{n_+},\rho_2\rangle_{\sp_{2n_+}\fq}
=\langle \pi_2,\Theta^{\epsilon'}_{n_+,n_-} (\rho_2)\rangle_{\o^{\epsilon}_{2(n_--1)+1}\fq}.
\]
Note that by Proposition \ref{o3} and Proposition \ref{w1} again, for any irreducible $\rho_2'\subset \Theta^{\epsilon'}_{n_+,n_-} (\rho_2)$,
\[
\rho_2'\subset I^{\o^\e_{2n_-}}(\tau\otimes\Theta^\e_{n_+-\ell,n_--\ell}(\pi_1)),
\]
hence 
\begin{equation}\label{rho2}
\rho_2'\in \cal{E}(\o^{\epsilon'}_{2n_-},\sigma'),
\end{equation}
 and
$
n_-<n+1.
$
 By (\ref{pi2}), (\ref{rho2}) and induction on $n$, one has
 \[
 \langle \pi_2,\Theta^{\epsilon'}_{n_+,n_-} (\rho_2)\rangle_{\o^{\epsilon}_{2(n_--1)+1}}=0,
 \]
which proves Case (A).
\end{proof}

As an immediate consequence of Proposition \ref{so1}, Proposition \ref{o2}, Lemma \ref{4} and first occurrence index of unipotent cuspidal representations, we see that if $\pi'$ is not the unique unipotent cuspidal representation specified in
Theorem \ref{o1}, then $m(\pi, \pi')=0$.

\subsection{Non-vanishing result}
To finish the proof of  Theorem \ref{o1}, by Corollary \ref{o4} it remains to prove the following result.

\begin{proposition}\label{oo}
(i) Let $\pi$ be an irreducible unipotent cuspidal representation of $\o^\epsilon_{2k(k+1)+1}\fq$, and $\tau_2$ be an irreducible cuspidal non-selfdual representation of $\GGL_k\fq$. If $\pi'$ is the irreducible unipotent cuspidal representation of $\o^{\epsilon(k)}_{k^2}\fq$ such that
\[
(n^\epsilon(\pi)-k(k+1))(n^{\epsilon(k)}(\pi')-k^2)> 0,
\]
 then
\[
\langle I^{\rm{O}^{\epsilon(k)}_{2k(k+1)}}_{\GGL_{k}\times \rm{O}^\e_{2k^2}}(\tau_2\otimes\pi'),\pi\rangle _{\rm{O}^{\epsilon(k)}_{2k(k+1)}(\Fq)}=1.
\]

(ii) Let $\pi$ be an irreducible unipotent cuspidal representation of $\o^{\epsilon(k)}_{2k^2}\fq$, and $\tau_2$ be an irreducible cuspidal non-selfdual representation of $\GGL_{k-1}\fq$. If $\pi'$ is the  irreducible unipotent cuspidal representation of $\o^{\epsilon'}_{2k(k-1)+1}\fq$ such that
\[
(n^{\epsilon(k)}(\pi)-k)^2)(n^\e(\pi')-k(k-1))>0,
\]
 then
\[
\langle I^{\rm{O}^\e_{2k^2-1}}_{\GGL_{k-1}\times \rm{O}_{2k(k-1)+1}}(\tau_2\otimes\pi'),\pi\rangle _{\rm{O}^\e_{2k^2-1}(\Fq)}=1.
\]
\end{proposition}

\begin{proof}
We will only prove (i), by induction on $k$. The proof of (ii) is similar and will be left to
the reader.

By Theorem \ref{spthetalift}, one has $n^\epsilon(\pi)=(k+1)^2$ or $k^2$. Note that if $n^\epsilon(\pi)=(k+1)^2$, then by the conservation relation for cuspidal representations given in \cite[Theorem 12.3]{P1}, one has
\[
n^\epsilon(\pi\otimes\rm{sgn})=k^2.
\]
 On the other hand  the conservation relation implies that
\begin{align*}
(n^\epsilon(\pi)-k(k+1))(n^{\epsilon(k)}(\pi')-k^2)> 0 \Longleftrightarrow (n^\epsilon(\pi\otimes\rm{sgn})-k(k+1))(n^{\epsilon(k)}(\pi'\otimes\rm{sgn})-k^2)> 0
\end{align*}
and 
\[
\langle I^{\rm{O}^{\epsilon(k)}_{2k(k+1)}}_{\GGL_{k}\times \rm{O}^{\epsilon(k)}_{2k^2}}(\tau_2\otimes\pi'),\pi\rangle _{\rm{O}^{\epsilon(k)}_{2k(k+1)}(\Fq)}=\langle I^{\rm{O}^{\epsilon(k)}_{2k(k+1)}}_{\GGL_{k}\times \rm{O}^{\epsilon(k)}_{2k^2}}(\tau_2\otimes(\pi'\otimes\rm{sgn})),\pi\otimes\rm{sgn}\rangle _{\rm{O}^{\epsilon(k)}_{2k(k+1)}(\Fq)}.
\]
Hence it suffices to prove the case that $n^\epsilon(\pi)=k^2$. Note that in this case $n^{\epsilon(k)}(\pi')=k(k-1)$.

Consider the see-saw diagram
\[
\setlength{\unitlength}{0.8cm}
\begin{picture}(20,5)
\thicklines
\put(6.2,4){$\sp_{2k^2}\times \sp_{2k^2}$}
\put(7,1){$\sp_{2k^2}$}
\put(12.4,4){$\o^\epsilon_{2k(k+1)+1}$}
\put(11.4,1){$\o^{\epsilon(k)}_{2k(k+1)}\times \o^{\epsilon''}_1$}
\put(7.7,1.5){\line(0,1){2.1}}
\put(12.8,1.5){\line(0,1){2.1}}
\put(8,1.5){\line(2,1){4.2}}
\put(8,3.7){\line(2,-1){4.2}}
\end{picture}
\]
where $\epsilon'':=\epsilon(k)\cdot\epsilon$.
 By Theorem \ref{spthetalift}, there is an irreducible cuspidal $\theta$-representation $\pi^\theta_{k,i}$ of $\sp_{2k^2}\fq$, $i\in\{\alpha,\beta\}$, such that
 \[
\langle\pi, I^{\rm{O}^{\epsilon(k)}_{2k(k+1)}}_{ \GGL_{k}\times \o^{\epsilon(k)}_{2k^2}}(\tau_2\otimes\pi')\rangle_{\o^{\epsilon(k)}_{2k(k+1)}}
=\langle\Theta^\epsilon_{k^2,k(k+1)}(\pi^\theta_{k,i}), I^{\rm{O}^{\epsilon(k)}_{2k(k+1)}}_{ \GGL_{k}\times \o^{\epsilon(k)}_{2k^2}}(\tau_2\otimes\pi')\rangle_{\o^{\epsilon(k)}_{2k(k+1)}}.
\]
By Mackey formula (c.f. \cite[Proposition 9.2.4]{C}),
\[ I^{\rm{O}^{\epsilon(k)}_{2k(k+1)}}_{ \GGL_{k}\times \o^\e_{2k^2}}(\tau_2\otimes\pi')
\textrm{ and }
I^{\sp_{2k^2}}_{ \GGL_{k}\times \sp_{2k(k-1)}}(\tau_2\otimes\Theta^{\epsilon(k)}_{k^2,k-1}(\pi'))=I^{\sp_{2k^2}}_{ \GGL_{k}\times \sp_{2k(k-1)}}(\tau_2\otimes\pi_{\sp_{2k(k-1)}})
\]
are irreducible, where $\pi_{\sp_{2k(k-1)}}$ is the unique unipotent cuspidal representation of $\sp_{2k(k-1)}\fq$.
By Proposition \ref{w1}, one has
\begin{align*}
&
\langle\Theta^\epsilon_{k^2,k(k+1)}(\pi^\theta_{k,i}), I^{\rm{O}^{\epsilon(k)}_{2k(k+1)}}_{ \GGL_{k}\times \o^{\epsilon(k)}_{2k^2}}(\tau_2\otimes\pi')\rangle_{\o^{\epsilon(k)}_{2k(k+1)}}\\
=&\langle\pi^\theta_{k,i},\Theta^{\epsilon(k)}_{k(k+1),k^2}( I^{\rm{O}^{\epsilon(k)}_{2k(k+1)}}_{ \GGL_{k}\times \o^{\epsilon(k)}_{2k^2}}(\tau_2\otimes\pi'))\otimes\omega^{\epsilon''}_{k^2}\rangle_{\sp_{2k^2}\fq}\\
=&\langle\pi^\theta_{k,i}, I^{\sp_{2k^2}}_{ \GGL_{k}\times \sp_{2k(k-1)}}(\tau_2\otimes\pi_{\sp_{2k(k-1)}}))\otimes\omega^{\epsilon''}_{k^2}\rangle_{\sp_{2k^2}\fq}.
\end{align*}

 To evaluate the last  multiplicity, consider another see-saw diagram
\[
\setlength{\unitlength}{0.8cm}
\begin{picture}(20,5)
\thicklines
\put(6.2,4){$\sp_{2k^2}\times \sp_{2k^2}$}
\put(7,1){$\sp_{2k^2}$}
\put(12.1,4){$\o^{\epsilon(k-1)}_{2(k(k-1)+1)}$}
\put(11.1,1){$\o^{-\epsilon}_{2k(k-1)+1}\times \o^{\epsilon_{-1}\cdot \epsilon''}_1$}
\put(7.7,1.5){\line(0,1){2.1}}
\put(12.8,1.5){\line(0,1){2.1}}
\put(8,1.5){\line(2,1){4.2}}
\put(8,3.7){\line(2,-1){4.2}}
\end{picture}
\]
Noting that $\overline{\omega^{\epsilon''}_{k^2}}\cong \omega^{\epsilon_{-1}\cdot \epsilon''}_{k^2}$, one has
\[
\begin{aligned}
&\langle\pi^\theta_{k,i}, I^{\sp_{2k^2}}_{ \GGL_{k}\times \sp_{2k(k-1)}}(\tau_2\otimes\pi_{\sp_{2k(k-1)}})\otimes\omega^{\epsilon''}_{k^2}\rangle_{\sp_{2k^2}\fq}\\
=&\langle\pi^\theta_{k,i}\otimes\omega^{\epsilon_{-1}\cdot \epsilon''}_{k^2}, I^{\sp_{2k^2}}_{ \GGL_{k}\times \sp_{2k(k-1)}}(\tau_2\otimes\pi_{\sp_{2k(k-1)}})\rangle_{\sp_{2k^2}\fq}\\
=&\langle\Theta^{-\epsilon}_{k(k-1),k^2}(\pi^{\eta}_{k-1})\otimes\omega^{\epsilon_{-1}\cdot \epsilon''}_{k^2}, I^{\sp_{2k^2}}_{ \GGL_{k}\times \sp_{2k(k-1)}}(\tau_2\otimes\pi_{\sp_{2k(k-1)}})\rangle_{\sp_{2k^2}\fq}\\
=&\langle\pi^{\eta}_{k-1},
\Theta^{\epsilon(k-1)}_{k^2,k(k-1)+1}(I^{\sp_{2k^2}}_{ \GGL_{k}\times \sp_{2k(k-1)}}(\tau_2\otimes\pi_{\sp_{2k(k-1)}}))
\rangle_{\o^{-\epsilon}_{2k(k-1)+1}\fq},
\end{aligned}
\]
where $\pi^{\eta}_{k-1}$ is the irreducible unipotent cuspidal representation of $\o^{\epsilon(k-1)}_{2k(k-1)+1}\fq$ such that
\[
n^{-\epsilon}(\pi^{\eta}_{k-1})=k^2.
\]
By Mackey formula and Proposition \ref{w1} again, the above multiplicity is equal to
\[
\langle\pi^{\eta}_{k-1}, I^{\o^{\epsilon(k-1)}_{2(k(k-1)+1)}}_{ \GGL_{k}\times \o^{-\e}_{2(k-1)^2}}(\tau_2\otimes\pi'^{\eta'}_{k-1})\rangle_{\o^{-\epsilon}_{2k(k-1)+1}\fq},
\]
where $ \pi'^{\eta'}_{k-1}$ is the irreducible unipotent cuspidal representations of $\o^{{\epsilon(k-1)}}_{2(k-1)^2}\fq$ such that
\[
n^{\epsilon(k-1)}(\pi'^{\eta'}_{k-1})=k(k-1).
\]
Applying Corollary \ref{o4}, this multiplicity is further reduced to 
\[
\langle\pi^{\eta}_{k-1}, I^{\o^{\epsilon(k-1)}_{2k(k-1)}}_{ \GGL_{k-1}\times \o^{-\e}_{2(k-1)^2}}(\tau_2'\otimes\pi'^{\eta'}_{k-1})\rangle_{\o^{-\epsilon}_{2k(k-1)}\fq},
\]
where $\tau_2'$ is an irreducible cuspidal non-selfdual representation of $\GGL_{k-1}\fq$.  
Since
\[
(n^{-\epsilon}(\pi^{\eta}_{k-1})-k(k-1))(n^{\epsilon(k-1)}(\pi'^{\eta'}_{k-1})-(k-1)^2)> 0,
\]
the proof is done by applying the induction hypothesis and verifying the initial cases which are fairly easy.
\end{proof}

\section{Fourier-Jacobi case of theorem \ref{main}} \label{sec6}

We have established the Bessel descents of unipotent cuspidal representations of finite orthogonal
groups. In this section we deduce the Fourier-Jacobi case from the Bessel case by the standard arguments of the theta correspondence and see-saw dual pairs, which are used in the proof of local Gan-Gross-Prasad conjecture (see \cite{GI, Ato}).

Recall that $\psi'$ is a nontrivial additive character of $\Fq$ not in the square class of $\psi$, so that
\[
\omega^+_N=\omega_{\rm{Sp}_{2N},\psi} \quad \rm{and}\quad \omega^-_N=\omega_{\rm{Sp}_{2N},\psi'}
\]
  are the Weil representations of the finite symplectic group $\rm{Sp}_{2N}(\Fq)$ corresponding to $\psi$ and $\psi'$ respectively. In general,  restricted to the dual pair $\sp_{2n}\fq\times\o_{2n'+1}^\epsilon\fq$ with $N=n(2n'+1)$, one has
\begin{equation}\label{omega}
\omega^\epsilon_{n,n',\psi}\cong \omega^{-\epsilon}_{n,n',\psi'}
\end{equation}
via the isomorphism $\o_{2n'+1}^\epsilon\cong \o_{2n'+1}^{-\epsilon}$.

To establish the Fourier-Jacobi descent, we again make the first reduction to the basic case.

\begin{proposition}\label{so2}
Let $\pi$ be an irreducible unipotent representation of $\sp_{2n}(\Fq)$, and $\pi'$ be an irreducible representation of $\sp_{2m}$ with $n > m$. Let $P$ be an $F$-stable maximal parabolic subgroup of $\sp_{2n}$ with Levi factor $\GGL_{n-m} \times \sp_{2m}$, and let $\tau$ be an irreducible cuspidal representation of $\GGL_{n-m}\fq$ which is nontrivial  if $n-m=1$. Then we have
\[
m_\psi(\pi, \pi')=\langle \pi\otimes\bar{\nu},\pi'\rangle_{H(\Fq)}=\langle  \pi\otimes \overline{\omega^+_n}, I_{P}^{\sp_{2n}}(\tau\otimes\pi')\rangle _{\sp_{2n}(\Fq)},
\]
where the data $(H,\nu)$ is given by (\ref{hnu'}).
\end{proposition}
Similar to Proposition \ref{so1}, the proof of Proposition \ref{so2} is an adaptation of that of \cite[Theorem 16.1]{GGP1}. Finally we prove the following Fourier-Jacobi case of Theorem \ref{main}.

\begin{theorem}
For the unique irreducible unipotent cuspidal representation $\pi_{\sp_{2k(k+1)}}$ of $\sp_{2k(k+1)}\fq$, one has  $\ell_0^\rm{FJ}(\pi_{\sp_{2k(k+1)}})=k$ and 
\[
 \CD^\rm{FJ}_{k, \psi}(\pi_{\sp_{2k(k+1)}})=\pi^\theta_{k,\alpha_k}, \quad  \CD^\rm{FJ}_{k, \psi'}(\pi_{\sp_{2k(k+1)}})=\pi^\theta_{k,\beta_k},
\]
where $(\alpha_k, \beta_k)=(\alpha, \beta)$ or $(\beta, \alpha)$ for $\epsilon_{-1}\cdot \epsilon(k)=+1$ or $-1$, respectively. 
\end{theorem}

\begin{proof}
Write $\pi=\pi_{\sp_{2k(k+1)}}$, and let $\pi'$ be an irreducible representation of $\sp_{2m}$, $m\leq k^2$.  Let $\tau$ be an irreducible cuspidal non-selfdual representation of $\GGL_{k(k+1)-m}\fq$. 
Put 
\[
\epsilon=\epsilon_{-1}\cdot \epsilon(k),
\]
and consider the see-saw diagram
\[
\setlength{\unitlength}{0.8cm}
\begin{picture}(20,5)
\thicklines
\put(5.6,4){$\sp_{2k(k+1)}\times \sp_{2k(k+1)}$}
\put(6.7,1){$\sp_{2k(k+1)}$}
\put(12.3,4){$\o^{\epsilon}_{2k^2+1}$}
\put(11.6,1){$\o^{\epsilon(k)}_{2k^2}\times \o^{\epsilon_{-1}}_1$}
\put(7.7,1.5){\line(0,1){2.1}}
\put(12.8,1.5){\line(0,1){2.1}}
\put(8,1.5){\line(2,1){4.2}}
\put(8,3.7){\line(2,-1){4.2}}
\end{picture}
\]
As before, we suppress various Levi subgroups from the parabolic induction.

$\bullet$ First suppose that $m<k^2$.

By Theorem \ref{even} and Proposition \ref{so2}, and noting that $\overline{\omega^+_n}\cong \omega_n^{\epsilon_{-1}}$, one has
\[
\begin{aligned}
&m_\psi(\pi,\pi')\\
=&\langle \pi\otimes\omega^{\epsilon_{-1}}_{k(k+1)},I^{\sp_{2k(k+1)}}(\tau\otimes \pi')\rangle_{\sp_{2k(k+1)}(\Fq)}\\
=&\langle \Theta^{\epsilon(k)}_{k^2,k(k+1)}(\pi^{-}_{k})\otimes\omega^{\epsilon_{-1}}_{k(k+1)},I^{\sp_{2k(k+1)}}(\tau\otimes \pi')\rangle_{\sp_{2k(k+1)}(\Fq)},
\end{aligned}
\]
where $\pi^-_k$ is one of the irreducible unipotent cuspidal representations of $\o_{2k^2}^{\epsilon(k)}\fq$.

For an irreducible $\rho'\subset I^{\sp_{2k(k+1)}}(\tau\otimes \pi')\rangle_{\sp_{2k(k+1)}(\Fq)}$, one has
\[
\begin{aligned}
&\langle \Theta^{\epsilon(k)}_{k^2,k(k+1)}(\pi^{-}_{k})\otimes\omega^{\epsilon_{-1}}_{k(k+1)},\rho'\rangle_{\sp_{2k(k+1)}(\Fq)}=\langle \pi^{-}_{k},\Theta^{\epsilon}_{k(k+1),k^2}(\rho')\rangle_{\o^{\epsilon(k)}_{2k^2}(\Fq)}.
\end{aligned}
\]
 By Proposition \ref{w1}, when $m\leq k$ one has $\Theta^{\epsilon}_{k(k+1),k^2}(\rho')=0$; when $k<m<k^2$, for any irreducible $\rho\subset \Theta^{\epsilon}_{k(k+1),k^2}(\rho')$, one has
\[
\rho\subset I^{\o_{2k^2+1}^{\epsilon}}((\chi\otimes\tau)\otimes \Theta^{\epsilon}_{m,m-k}(\pi')).
\]
It follows from Theorem \ref{o1} (ii) that
$
m_\psi(\pi,\pi')=0.
$
In the same manner, $m_{\psi'}(\pi,\pi')=0$ as well.

$\bullet$ Next suppose that $m=k^2$. In the above we have shown that
\begin{equation}\label{B->FJ}
 \langle \pi^{-}_{k}, I^{\o_{2k^2+1}^{\epsilon}}(\tau\otimes \Theta^{\epsilon}_{k^2,k(k-1)}(\pi'))\rangle_{\o^\epsilon_{2k^2}(\Fq)}=0 \Longrightarrow m_\psi(\pi,\pi')=0.
\end{equation}
Recall that $\epsilon=\epsilon_{-1}\cdot \epsilon(k)$. By Theorem \ref{o1} and Theorem \ref{spthetalift}, the first term of \eqref{B->FJ} is nonzero if and only if
$\pi'=\pi^\theta_{k, \alpha_k}$, where $\alpha_k=\alpha$ for $\beta$ for $\epsilon=+1$ or $-1$ respectively so that $n^{\epsilon}(\pi')=k(k-1)$. We need to show the converse of \eqref{B->FJ}, that is,
$m_\psi(\pi,\pi^\theta_{k,\alpha_k})\neq 0$.  We have
\[
m_\psi(\pi,\pi^\theta_{k,\alpha_k})
=\langle \Theta^{\epsilon(k)}_{k^2,k(k+1)}(\pi^{-}_{k})\otimes\omega^{\epsilon_{-1}}_{k(k+1)},I^{\sp_{2k(k+1)}}(\tau\otimes \pi^\theta_{k,\alpha_k})\rangle_{\sp_{2k(k+1)}(\Fq)}.
\]
Since $\tau$ is non-selfdual, $I^{\sp_{2k(k+1)}}(\tau\otimes \pi^\theta_{k,\alpha_k})$ is irreducible by Mackey formula. By Proposition \ref{w1}, one has
\begin{align*}
&\langle \Theta^{\epsilon(k)}_{k^2,k(k+1)}(\pi^{-}_{k})\otimes\omega^{\epsilon_{-1}}_{k(k+1)},I^{\sp_{2k(k+1)}}(\tau\otimes \pi^\theta_{k,\alpha_k})\rangle_{\sp_{2k(k+1)}(\Fq)}\\
=&\langle \pi^{-}_{k},\Theta^{\epsilon}_{k(k+1),k^2}(I^{\sp_{2k(k+1)}}(\tau\otimes \pi^\theta_{k,\alpha_k}))\rangle_{\o^{\epsilon(k)}_{2k^2}(\Fq)}\\
=&
\langle \pi^{-}_{k},I^{\o_{2k^2+1}^{\epsilon}}((\chi\otimes\tau)\otimes \Theta^{\epsilon}_{k^2,k(k-1)}(\pi^\theta_{k,\alpha_k}))\rangle_{\o^{\epsilon(k)}_{2k^2}(\Fq)},
\end{align*}
which is nonzero. Hence $m_\psi(\pi,\pi^\theta_{k,\alpha_k})\neq 0$ and it follows that
\[
\CQ^\rm{FJ}_{k,\psi}(\pi)=\pi^\theta_{k,\alpha_k}.
\]

We next turn to $m_{\psi'}(\pi,\pi')$. By (\ref{omega}), one has
\[
m_{\psi'}(\pi,\pi')
= \langle \pi\otimes\omega^{-\epsilon_{-1}}_{k(k+1)},I^{\sp_{2k(k+1)}}(\tau\otimes \pi')\rangle_{\sp_{2k(k+1)}(\Fq)}.
\]
Consider the see-saw diagram
\[
\setlength{\unitlength}{0.8cm}
\begin{picture}(20,5)
\thicklines
\put(5.6,4){$\sp_{2k(k+1)}\times \sp_{2k(k+1)}$}
\put(6.7,1){$\sp_{2k(k+1)}$}
\put(12.3,4){$\o^{-\epsilon}_{2k^2+1}$}
\put(11.6,1){$\o^{\epsilon(k)}_{2k^2}\times \o^{-\epsilon_{-1}}_1$}
\put(7.7,1.5){\line(0,1){2.1}}
\put(12.8,1.5){\line(0,1){2.1}}
\put(8,1.5){\line(2,1){4.2}}
\put(8,3.7){\line(2,-1){4.2}}
\end{picture}
\]
By Theorem \ref{spthetalift}, one has $n^{-\epsilon}(\pi^\theta_{k, \beta_k})=k(k-1)$, where $\beta_k=\beta$ or $\alpha$ for 
$\epsilon=+1$ or $-1$ respectively. Then one can similarly show that
\[
\CQ^\rm{FJ}_{k,\psi'}(\pi)=\pi^\theta_{k,\beta_k}.
\]
\end{proof}

\end{document}